\documentclass[a4paper,11pt]{article}

\usepackage{graphicx,amsmath,amsthm}

\theoremstyle{plain}
\newtheorem{theorem}{Theorem}
\newtheorem*{theorem*}{Theorem}
\newtheorem{proposition}[theorem]{Proposition}
\newtheorem*{proposition*}{Proposition}
\newtheorem{lemma}[theorem]{Lemma}
\newtheorem*{lemma*}{Lemma}

\newtheorem*{corollary*}{Corollary}
\theoremstyle{remark}

\newtheorem*{remark*}{Remark}
\theoremstyle{definition}
\newtheorem*{acknowledgment}{Acknowledgment}

\def\SQ{\mathbf{Q}}
\def\SC{\mathbf{C}}
\def\SP{\mathbf{P}}

\def\mult{\mathop{\rm mult}\nolimits}
\def\supp{\mathop{\rm supp}\nolimits}
\def\ms{\mathop{\overline{m}}\nolimits}
\def\KB{\mathop{\bar{\kappa}}\nolimits}

\def\Ch{\mathop{\rm Ch}\nolimits}

\def\NV{\mathop r\nolimits}
\def\TA{\mathbin{\ast}}
\def\TP#1{{\vphantom{#1}}^{\mathit{t}}{#1}}
\def\NS{\mathop{o}\nolimits}
\def\TW#1{t_{#1}}

\title{On a new class of rational cuspidal plane curves with two cusps}
\author{Keita Tono}
\date{}

\begin{document}
\maketitle
\def\thefootnote{}
\footnotetext{Supported by JSPS Grant-in-Aid for Scientific Research (22540040).}
\def\thefootnote{1}
\begin{abstract}
In this paper,
we consider rational cuspidal plane curves having exactly two cusps
whose complements have logarithmic Kodaira dimension two.
We classify such curves with the property that
the strict transforms of them
via the minimal embedded resolution of the cusps have
maximal self-intersection number.
\end{abstract}
%%%%%%%%%%%%%%%%%%%%%%%%%%%%%%%%%%%%%%%%%%%%%%%%%%%%%%%%%%%%%%%%%%%%%%%%%%%%%%%
\section{Introduction}
%%%%%%%%%%%%%%%%%%%%%%%%%%%%%%%%%%%%%%%%%%%%%%%%%%%%%%%%%%%%%%%%%%%%%%%%%%%%%%%
Let $C$ be an algebraic curve on $\SP^2=\SP^2(\SC)$.
A singular point of $C$ is said to be a \emph{cusp}
if it is a locally irreducible singular point.
We say that $C$ is \emph{cuspidal} (resp.~\emph{bicuspidal})
if $C$ has only cusps (resp.~two cusps) as its singular points.
For a cusp $P$ of $C$,
we denote the \emph{multiplicity sequence} of $(C,P)$ by $\ms_P(C)$,
or simply by $\ms_P$.
We usually omit the last 1's in $\ms_P$.
We use the abbreviation $m_k$
for a subsequence of $\ms_P$ consisting of $k$ consecutive $m$'s.
For example, $(2_k)$ means an $A_{2k}$ singularity.
The set of the multiplicity sequences of the cusps of
a cuspidal plane curve $C$
will be called the \emph{numerical data} of $C$.
For example, the rational quartic with three cusps has the numerical data
$\{(2),(2),(2)\}$.
We denote by $\KB=\KB(\SP^2\setminus C)$
the logarithmic Kodaira dimension
of the complement $\SP^2\setminus C$.

Suppose that $C$ is rational and bicuspidal.
By \cite{Wak,Ts}, we have $\KB\ge 1$.
Let $C'$ denote the strict transform of $C$
via the minimal embedded resolution of the cusps of $C$.
We characterize rational bicuspidal plane curves $C$ with $\KB=1$
by $(C')^2$ in the following way.
\begin{theorem}\label{thm0}
If $C$ is a rational bicuspidal plane curve, then $(C')^2\le 0$.
Moreover,  $(C')^2=0$ if and only if $\KB=1$.
\end{theorem}

We next consider rational bicuspidal plane curves $C$ with $(C')^2=-1$.
\begin{theorem}\label{thm1}
The numerical data of
a rational bicuspidal plane curve $C$ with $(C')^2=-1$
coincides with one of those in the following table, where $a$ is a positive integer.
\begin{center}
\def\arraystretch{1.5}
\begin{tabular}{c|l|l}
No. & \multicolumn{1}{|c}{Numerical data} & \multicolumn{1}{|c}{Degree} \\ \hline
1 & $\{(ab+b-1,ab-1,b_{a-1},b-1),$ $((ab)_2,b_{a})\}$ \hfill ($b\ge 2$) & $2ab+b-1$ \\
2 & $\{(ab+b,ab,b_{a}),$ $((ab+1)_2,b_{a})\}$ \hfill ($b\ge 2$) & $2ab+b+1$ \\
3 & $\{(ab+1,ab-b+1,b_{a-1}),$ $((ab)_2,b_{a})\}$ \hfill ($b\ge 3$) & $2ab+1$ \\
4 & $\{(ab+b,ab,b_{a}),$ $((ab+b-1)_2,b_a,b-1)\}$ \hfill ($b\ge 3$) & $2ab+2b-1$
\end{tabular}
\def\arraystretch{1}
\end{center}
Conversely,
for a given numerical data in the above table,
there exists a rational cuspidal plane curve
having that data.
\end{theorem}

In \cite{fe},
many sequences of rational bicuspidal plane curves
were constructed.
The numerical data of the curves with $(C')^2=-1$ among them
coincide with the data 1, 2 and 3 with $a=1$ in Theorem~\ref{thm1}.
%
%
%
%
%%%%%%%%%%%%%%%%%%%%%%%%%%%%%%%%%%%%%%%%%%%%%%%%%%%%%%%%%%%%%%%%%%%%%%%%%%%%%%%
\section{Preliminaries}
%%%%%%%%%%%%%%%%%%%%%%%%%%%%%%%%%%%%%%%%%%%%%%%%%%%%%%%%%%%%%%%%%%%%%%%%%%%%%%%
Let $D$ be a divisor on a smooth surface $V$,
$\varphi:V'\rightarrow V$ a composite of
successive blow-ups
and $B\subset V'$ a divisor.
We say that
$\varphi$ \emph{contracts} $B$ to $D$,
or simply that $B$ \emph{shrinks to} $D$
if
$\varphi(\supp{B})=\supp{D}$ and
each center of blow-ups of $\varphi$
is on $D$ or one of its preimages.
Let $D_1,\ldots,D_r$ be the irreducible components of $D$.
We call $D$ an \emph{SNC-divisor} if
$D$ is a reduced effective divisor,
each $D_i$ is smooth,
$D_iD_j\le 1$ for distinct $D_i,D_j$,
and $D_i\cap D_j\cap D_k=\emptyset$ for distinct $D_i,D_j,D_k$.

Assume that $D$ is an SNC-divisor
and that each $D_i$ is projective.
Let $\Gamma=\Gamma(D)$ denote the dual graph of $D$.
We give the vertex corresponding to a component $D_i$
the weight $D_i^2$.
We sometimes do not distinguish between $D$
and its weighted dual graph $\Gamma$.
We use the following notation and terminology
(cf.~\cite[Section 3]{fu} and \cite[Chapter 1]{mits}).
A blow-up at a point $P\in D$
is said to be \emph{sprouting} (resp.~\emph{subdivisional})
\emph{with respect to} $D$
if $P$ is a smooth point (resp.~node) of $D$.
We also use this terminology
for the case in which $D$ is a point.
By definition, the blow-up is subdivisional in this case.
A component $D_i$ is called a \emph{branching component} of $D$
if $D_i(D-D_i)\ge 3$.

Assume that $\Gamma$ is connected and linear.
In cases where $r>1$,
the weighted linear graph $\Gamma$ together with
a direction from an endpoint to the other
is called a \emph{linear chain}.
By definition,
the empty graph $\emptyset$
and a weighted graph consisting of a single vertex without edges
are linear chains.
If necessary, renumber $D_1,\ldots,D_r$ 
so that the direction of the linear chain $\Gamma$ is from $D_1$ to $D_r$
and $D_iD_{i+1}=1$ for $i=1,\ldots,r-1$.
We denote $\Gamma$ by $[-D_1^2,\ldots,-D_r^2]$.
We sometimes write $\Gamma$ as $[D_1,\ldots,D_r]$.
The linear chain is called \emph{rational} if every $D_i$ is rational.
In this paper, we always assume that every linear chain is rational.
The linear chain $\Gamma$ is called \emph{admissible} 
if it is not empty and $D_i^2\le -2$ for each $i$.
Set $\NV(\Gamma)=r$.
We define
the \emph{discriminant} $d(\Gamma)$ of $\Gamma$
as the determinant of the $r\times r$ matrix $(-D_i D_j)$.
We set $d(\emptyset)=1$.

Let $A=[a_1,\ldots,a_r]$ be a linear chain.
We use the following notation if $A\ne\emptyset$:
\[
\TP{A}:=[a_r,\ldots,a_1],\ 
\overline{A}:=[a_2,\ldots,a_r],\ 
\underline{A}:=[a_1,\ldots,a_{r-1}].
\]
The discriminant $d(A)$ has the following properties (\cite[Lemma 3.6]{fu}).
\begin{lemma}\label{lem:det1}
Let $A=[a_1,\ldots,a_r]$ be a linear chain.
\begin{enumerate}
\item[\textnormal{(i)}]
If $r>1$, then
$d(A)=a_1 d(\overline{A})-d(\overline{\overline{A}})=d(\TP{A})=a_r d(\underline{A})-d(\underline{\underline{A}})$.
\item[\textnormal{(ii)}]
If $r>1$, then
$d(\overline{A})d(\underline{A})-d(A)d(\underline{\overline{A}})=1$.
\item[\textnormal{(iii)}]
If $A$ is admissible,
then $\gcd(d(A),d(\overline{A}))=1$ and $d(A)>d(\overline{A})>0$.
\end{enumerate}
\end{lemma}
%
%%%%%%%%%%%%%%%%%%%%%%%%%%%%%%%%%%%%%%%%%%%%%%%%%%%%%%%%%%%%%%%%%%%%%%%%%%%%%%%
%\subsection{The inductance of a linear chain}
%%%%%%%%%%%%%%%%%%%%%%%%%%%%%%%%%%%%%%%%%%%%%%%%%%%%%%%%%%%%%%%%%%%%%%%%%%%%%%%

Let $A=[a_1,\ldots,a_r]$ be an admissible linear chain.
The rational number $e(A):=d(\overline{A})/d(A)$
is called the \emph{inductance} of $A$.
By \cite[Corollary 3.8]{fu}, the function
$e$ defines a one-to-one correspondence between the
set of all the admissible linear chains and the set of rational numbers
in the interval $(0,1)$.
For a given admissible linear chain $A$,
the admissible linear chain $A^{\ast}:=e^{-1}(1-e(\TP{A}))$ is called
the \emph{adjoint} of $A$ (\cite[3.9]{fu}).
Admissible linear chains and their adjoints have the following properties
(\cite[Corollary 3.7, Proposition 4.7]{fu}).
\begin{lemma}\label{lem:indf}
Let $A$ and $B$ be admissible linear chains.
\begin{enumerate}
\item[\textnormal{(i)}]
If $e(A)+e(B)=1$, then $d(A)=d(B)$ and $e(\TP{A})+e(\TP{B})=1$.
\item[\textnormal{(ii)}]
We have $A^{\ast\ast}=A$, $\TP{(A^{\ast})}=(\TP{A})^{\ast}$ and
$d(A)=d(A^{\ast})=d(\overline{A^{\ast}})+d(\underline{A})$.
\item[\textnormal{(iii)}]
The linear chain $[A,1,B]$ shrinks to $[0]$
if and only if $A=B^{\ast}$.
\end{enumerate}
\end{lemma}

For integers $m$, $n$ with $n\ge 0$, we define 
$[m_n]=[\overbrace{m,\ldots,m}^n]$, $\TW{n}=[2_n]$.
For non-empty linear chains $A=[a_1,\ldots,a_r]$, $B=[b_1,\ldots,b_s]$,
we write
$A\TA B=[\underline{A},a_r+b_1-1,\overline{B}]$,
$A^{\ast n}=\overbrace{A\TA\cdots\TA A}^n$, where $n\ge 1$.
We remark that $(A\TA B)\TA C=A\TA(B\TA C)$
for non-empty linear chains $A$, $B$ and $C$.
By using Lemma~\ref{lem:det1} and Lemma~\ref{lem:indf},
we can show the following lemma.
\begin{lemma}\label{lem:adj}
Let $A=[a_1,\ldots,a_r]$ be an admissible linear chain.
\begin{enumerate}
\item[(i)]
For a positive integer $n$, we have $[A,n+1]^{\ast}=\TW{n}\TA A^{\ast}$.
\item[(ii)]
We have $A^{\ast}=\TW{a_r-1}\TA\cdots\TA\TW{a_1-1}$.
\item[(iii)]
If there exist positive integers $m$, $n$ such that
$[A,m+1]=[n+1,A]$
(resp.~$A\TA\TW{m}=\TW{n}\TA A$),
then $m=n$,
$a_1=\cdots=a_r=n+1$
(resp.~$A=\TW{n}^{\ast\NV(A^{\ast})}$).
\end{enumerate}
\end{lemma}

We will use the following lemma
(\cite[Corollary 8]{to:orev}).
\begin{lemma}\label{lem:bu}
Let $a$ be a positive integer and $A$ an admissible linear chain.
Let $B$ be a linear chain which is empty or admissible.
Assume that
a composite $\pi$ of blow-downs contracts
$[A,1,B]$ to $[a]$
and that $[a]$ is the image of $A$ under $\pi$.
\begin{enumerate}
\item[(i)]
The linear chain $[a]$ is the image of the first curve of $A$.
There exits a positive integer $n$ such that
$A^{\ast}=[B,n+1,\TW{a-1}]$.
Moreover,
$A=[a]\TA\TW{n}\TA B^{\ast}$
if $B\ne\emptyset$.
\item[(ii)]
The first $n$ blow-ups of $\pi$ are sprouting and
the remaining ones are subdivisional with respect to
$[a]$ or its preimages.
The composite of the subdivisional blow-ups contracts
$[A,1,B]$ to $[[a]\TA\TW{n},1]$.
\item[(iii)]
The exceptional curve of each blow-up of $\pi$
is a unique ($-1$)-curve in the preimage of $[a]$.
\end{enumerate}
Conversely,
$[[a]\TA\TW{n}\TA B^{\ast},1,B]$ shrinks to $[a]$
for given positive integers $a$, $n$ and an admissible linear chain $B$.
\end{lemma}
%%%%%%%%%%%%%%%%%%%%%%%%%%%%%%%%%%%%%%%%%%%%%%%%%%%%%%%%%%%%%%%%%%%%%%%%%%%%%%%
\subsection{%
Resolution of a cusp}\label{sec:cres}
%%%%%%%%%%%%%%%%%%%%%%%%%%%%%%%%%%%%%%%%%%%%%%%%%%%%%%%%%%%%%%%%%%%%%%%%%%%%%%%
Let $(C,P)$ be a curve germ on a smooth surface $V$.
Suppose that $(C,P)$ is a cusp.
Let $\sigma:V'\rightarrow V$ be the minimal embedded resolution
of $(C,P)$.
That is,
$\sigma$ is the composite of the shortest sequence
of blow-ups such that
the strict transform $C'$ of $C$ intersects $\sigma^{-1}(P)$ transversally.
Let
\(
V'=V_n\stackrel{\sigma_{n-1}}{\longrightarrow}V_{n-1}
\longrightarrow\cdots\longrightarrow
V_2\stackrel{\sigma_1}{\longrightarrow}
V_1\stackrel{\sigma_0}{\longrightarrow}V_0=V
\)
be the blow-ups of $\sigma$.
The following lemma follows from the assumptions that
$(C,P)$ is a cusp and $\sigma$ is minimal.
\begin{lemma}\label{lem:cres0}
For $i\ge1$, the strict transform of $C$ on $V_i$
intersects $(\sigma_0\circ\cdots\circ\sigma_{i-1})^{-1}(P)$
in one point, which is on the exceptional curve of $\sigma_{i-1}$.
The point of intersection is the center of $\sigma_{i}$ if $i<n$.
\end{lemma}

Let $D_0$ denote the exceptional curve of the last blow-up of $\sigma$.
\begin{lemma}[{\cite[Lemma 11]{to:orev}}]\label{lem:cres}
The following assertions hold.
\begin{enumerate}
\item[(i)]
The dual graph of $\sigma^{-1}(C)$ has the following shape,
where
$g\ge1$ and
$A_1$ contains the exceptional curve of $\sigma_0$ by definition.
\begin{center}
\includegraphics{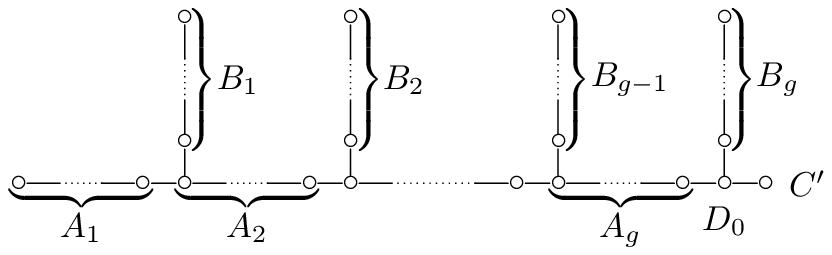}
\end{center}
We number the irreducible components $A_{i,1}, A_{i,2},\ldots$ of $A_i$
(resp.~$B_{i,1}$, $B_{i,2},\ldots$ of $B_i$)
from the left-hand side to the right
(resp.~the bottom to the top) in the above figure.
With these directions and the weights $A_{i,1}^2,A_{i,2}^2,\ldots$,
$B_{i,1}^2,B_{i,2}^2,\ldots$,
we regard $A_i,B_i$ as linear chains.
\item[(ii)]
The morphism $\sigma$ can be written as
$\sigma=\sigma_0\circ\rho_1'\circ\rho_1''\circ\cdots\circ\rho_g'\circ\rho_g''$,
where each $\rho_i'$ (resp.~$\rho_i''$) consists of
sprouting (resp.~subdivisional) blow-ups
of $\sigma$ with respect to preimages of $P$.
\item[(iii)]
The morphisms $\rho_i:=\rho_i'\circ\rho_i''$
have the following properties.
\begin{enumerate}
\item[(a)]
For $j<i$,
$\rho_i$ does not change the linear chains $A_j,B_j$.
\item[(b)]
For each $i$,
$\rho_{i}\circ\dots\circ\rho_{g}$ maps $A_{i,1}$ to a ($-1$)-curve.
\item[(c)]
$\rho_g$ contracts the linear chain $A_g+D_0+B_g$ to
the ($-1$)-curve $\rho_g(A_{g,1})$.
For $i<g$,
$\rho_{i}$ contracts
the linear chain $(\rho_{i+1}\circ\dots\circ\rho_{g})(A_{i}+A_{i+1,1}+B_{i})$
to the ($-1$)-curve $(\rho_{i}\circ\dots\circ\rho_{g})(A_{i,1})$.
\end{enumerate}
\end{enumerate}
\end{lemma}

We regard $A_i$ and $B_i$ as linear chains
in the same way as in Lemma~\ref{lem:cres} (i).
By Lemma~\ref{lem:cres0},
these linear chains are admissible.
Let $\NS_i$ denote the number of the blow-ups in $\rho_i'$.
The following proposition follows from Lemma~\ref{lem:bu}.
\begin{proposition}\label{prop:cres}
The following assertions hold for $i=1,\ldots,g$.
\begin{enumerate}
\item[(i)]
We have $A_i=\TW{\NS_i}\TA B_i^{\ast}$, $A_i^{\ast}=[B_i,\NS_i+1]$.
\item[(ii)]
The linear chain $A_i$ contains an irreducible component $E$ with $E^2\le-3$.
\end{enumerate}
\end{proposition}
%
%%%%%%%%%%%%%%%%%%%%%%%%%%%%%%%%%%%%%%%%%%%%%%%%%%%%%%%%%%%%%%%%%%%%%%%%%%%%%%%
\subsection{%
The characteristic sequence of a cusp}\label{sec:char}
%%%%%%%%%%%%%%%%%%%%%%%%%%%%%%%%%%%%%%%%%%%%%%%%%%%%%%%%%%%%%%%%%%%%%%%%%%%%%%%
Let the notation be as in the previous subsection.
Put $\alpha_0=\mult_P C$.
We take local coordinates $(x,y)$ of $V$ around $P=(0,0)$ such that
the germ $(C,P)$ has a local parameterization:
\[
x=t^{\alpha_0},\ 
y=\sum_{i=\alpha_1}^{\infty}c_i t^i
\quad
\textnormal{($c_{\alpha_1}\ne0$, $\alpha_1>\alpha_0$, $\alpha_1 \not\equiv 0 \pmod{\alpha_0}$)}.
\]
The \emph{characteristic sequence} of $(C,P)$,
which is denoted by $\Ch_{P}=\Ch_{P}(C)$,
is a sequence $(\alpha_0,\alpha_1,\ldots,\alpha_k)$ of positive integers
defined by the following conditions.
\begin{enumerate}
\item[\textnormal{(i)}]
  $\gcd(\alpha_0,\ldots,\alpha_{k})=1$.
\item[\textnormal{(ii)}]
  If $\gcd(\alpha_0,\ldots,\alpha_{i-1})>1$,
  then $\alpha_i$ is the smallest $j$ such that $c_j \ne 0$ and that
  $\gcd(\alpha_0,\ldots,\alpha_{i-1})>\gcd(\alpha_0,\ldots,\alpha_{i-1},j)$.
\end{enumerate}
The multiplicity sequence of $P$ is determined by $\Ch_{P}$ as follows.
Put $\gamma_i=\alpha_i-\alpha_{i-1}$ for $i=1,\ldots,k$.
Perform the Euclidean algorithm for $i=1,\ldots,k$:
\[
\vbox{\ialign{\hfil$#$&$#$\hfil\quad&#\hfil\cr
\gamma_i&{}=a_{i,1}m_{i,1}+m_{i,2}&$(0<m_{i,2}<m_{i,1})$,\cr
m_{i,1}&{}=a_{i,2}m_{i,2}+m_{i,3}&$(0<m_{i,3}<m_{i,2})$,\cr
&\cdots&\hfil$\cdots$\cr
m_{i,n_i-2}&{}=a_{i,n_i-1}m_{i,n_i-1}+m_{i,n_i}&$(0<m_{i,n_i}<m_{i,n_i-1})$,\cr
m_{i,n_i-1}&{}=a_{i,n_i}m_{i,n_i},&\cr
}}
\]
where $m_{1,1}=\alpha_0$ and $m_{i+1,1}=m_{i,n_i}$.
Note that $a_{i,n_i}>1$, $n_i>1$,
and that $a_{i,j}>0$ if $j>1$ but $a_{i,1}\ge 0$ for each $i$.
The multiplicity sequence of $P$ is given by
\[
(\alpha_0,
\overbrace{m_{1,1},\ldots,m_{1,1}}^{a_{1,1}},\ldots,
\overbrace{m_{i,j},\ldots,m_{i,j}}^{a_{i,j}},\ldots,
\overbrace{1,\ldots,1}^{a_{k,n_k}}).
\]
Conversely,
$\Ch_{P}$ is determined from $\ms_P$ by the above relation.
See \cite[p.516, Theorem 12]{bk} for details,
where $\gamma_1$ is defined as $\gamma_1=\alpha_1$.
We remark that the \emph{Puiseux pairs}
$(q_1,p_1),\ldots,(q_k,p_k)$
of $(C,P)$ are computed
from $\Ch_P$ by the relations:
\[
\alpha_0=q_1\cdots q_k,\,\,
\frac{\alpha_i}{\alpha_0}=\frac{p_i}{q_1\cdots q_i},\,\,
\gcd(q_i,p_i)=1\text{ for $i=1,\ldots,k$.}
\]

We next describe the relation between
the multiplicity sequence determined by $\Ch_{P}$
and the linear chains $A_i$, $B_i$.
\begin{proposition}[{cf.~\cite[p.524, Theorem 15]{bk}}]\label{prop:ms}
We have the following relations between
the multiplicity sequence
$(m_{1,1},(m_{1,1})_{a_{1,1}},\ldots,(m_{k,n_k})_{a_{k,n_k}})$
and
$A_1,B_1,\ldots,A_g,B_g$.
In particular $g=k$.
\begin{enumerate}
\item[(i)]
If $n_i$ is an odd number,
then
\begin{eqnarray*}
&A_i=\TW{a_{i,1}+1}\TA[a_{i,2}]\TA\cdots\TA\TW{a_{i,n_i-2}+1}\TA[a_{i,n_i-1}]\TA\TW{a_{i,n_i}},\\
&B_i=[a_{i,n_i}]\TA\TW{a_{i,n_i-1}+1}\TA\cdots\TA[a_{i,5}]\TA\TW{a_{i,4}+1}\TA[a_{i,3}]\TA\TW{a_{i,2}},
\end{eqnarray*}
where we interpret $A_i$, $B_i$ as
$A_i=\TW{a_{i,1}+1}\TA[a_{i,2}]\TA\TW{a_{i,3}}$,
$B_i=[a_{i,3}]\TA\TW{a_{i,2}}$ when $n_i=3$.
\item[(ii)]
If $n_i$ is an even number,
then
\begin{eqnarray*}
&A_i=\TW{a_{i,1}+1}\TA[a_{i,2}]\TA\cdots\TA\TW{a_{i,n_i-1}+1}\TA[a_{i,n_i}],\\
&B_i=\TW{a_{i,n_i}}\TA[a_{i,n_i-1}]\TA\TW{a_{i,n_i-2}+1}\TA\cdots\TA[a_{i,5}]\TA\TW{a_{i,4}+1}\TA[a_{i,3}]\TA\TW{a_{i,2}},
\end{eqnarray*}
where we interpret $A_i$, $B_i$ as
$A_i=\TW{a_{i,1}+1}\TA[a_{i,2}]$,
$B_i=\TW{a_{i,2}-1}$ when $n_i=2$.
\end{enumerate}
We have the weighted dual graphs in Figure~\ref{fig-ch} of $A_i$ and $B_i$,
where the vertices are ordered from the left-hand side to the right,
and $\ast$ (resp.~$\bullet$) denotes
a $(-1)$-curve (resp.~$(-2)$-curve).
\end{proposition}
\begin{figure}
\begin{center}
\includegraphics{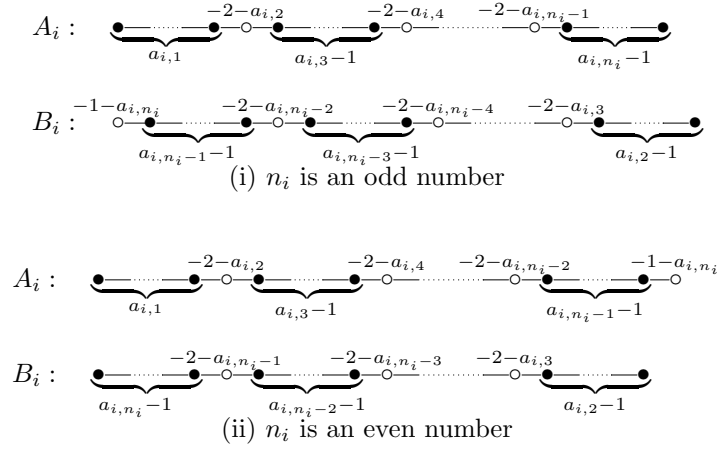}
\end{center}
\caption{The weighted dual graphs of $A_i$ and $B_i$}\label{fig-ch}
\end{figure}
In order to prove Proposition~\ref{prop:ms},
we need Lemma~\ref{lem:ch1} and Lemma~\ref{lem:ch2} below.
Let
\(
V'=V_n\stackrel{\sigma_{n-1}}{\longrightarrow}V_{n-1}
\longrightarrow\cdots\longrightarrow
V_2\stackrel{\sigma_1}{\longrightarrow}
V_1\stackrel{\sigma_0}{\longrightarrow}V_0=V
\)
be the blow-ups of the minimal embedded resolution $\sigma$
of the cusp $P$
as in the previous subsection.
For $i>j$,
put $\tau_{i,j}=\sigma_{j}\circ\sigma_{j+1}\circ\cdots\circ\sigma_{i-1}:V_i\rightarrow V_j$.
Let $E_i$ denote the exceptional curve of $\sigma_{i-1}$.
We use the same symbol to denote the strict transforms of $E_i$.
Let $(C_i,P_i)$ denote the strict transform of the curve germ $(C,P)$ on $V_i$,
where $C_i\cap E_i=\{P_i\}$.
Write $\ms_{P}(C)$ as $\ms_{P}(C)=(m_0,m_1,\ldots)$.

\begin{lemma}[{cf.~\cite[Lemma 1.3]{fz:dd2}}]\label{lem:ch1}
Suppose $m_0=\cdots=m_{q-1}$.
\begin{enumerate}
\item[(i)]
$(C_q E_q)_{P_q}=m_0$ and $(C_q E_i)_{P_q}=0$ for each $i\ne q$.
\item[(ii)]
The dual graph of $\tau_{q,0}^{-1}(P)$ is linear.
We have
\[
\tau_{q,0}^{-1}(P)=[E_1, E_2, \ldots, E_q]=[\TW{q-1},1].
\]
\end{enumerate}
\end{lemma}
\begin{proof}
The assertion (i) follows from \cite[Lemma 1.3]{fz:dd2}.
We prove the assertion (ii) by induction on $q$.
The assertion is clear if $q=1$.
Assume $q>1$.
We have $\ms_{P_1}=(m_1,m_2,\ldots)$.
By the induction hypothesis,
the dual graph of
$\tau_{q,1}^{-1}(P_1)$ is linear and
$\tau_{q,1}^{-1}(P_1)=[E_2,\ldots,E_q]=[\TW{q-2},1]$.
By (i),
the center of $\sigma_1$ is on $E_1$, while that of $\sigma_i$ is not
for $i\ge 2$.
This means that $E_1^2=-2$ on $V_q$ and that
$E_1$ intersects only $E_2$ among $E_2,\ldots,E_q$.
\end{proof}
\begin{lemma}[{cf.~\cite[Lemma 1.4]{fz:dd2}}]\label{lem:ch2}
Let $l$ be a projective curve on $V$ which is smooth at $P$.
Let $l_i$ denote the strict transform of $l$ on $V_i$.
Write $(C l)_P=q m_0+r$,
where $1\le q$, $0\le r<m_0$.
\begin{enumerate}
\item[(i)]
  We have $m_0=\cdots=m_{q-1}$.
  Moreover, $m_{q}=r$ if $r>0$.
\item[(ii)]
  We have
  $(C_q l_q)_{P_q}=r$, $l_q^2=l^2-q$,
  $\tau_{q,0}^{-1}(l)=E_1+\cdots+E_q+l_q$,
  $E_q l_q=1$ and $E_1 l_q=\cdots=E_{q-1} l_q=0$.
\end{enumerate}
\end{lemma}
\begin{proof}
The assertion (i) follows from \cite[Lemma 1.4]{fz:dd2}.
We prove the assertion (ii) by induction on $q$.
On $V_1$, we have $(C_1 l_1)_{P_1}=(q-1)m_0+r$,
$l_1^2=l^2-1$ and $\sigma_0^{-1}(l)=E_1+l_1$.
So the assertion is clear if $q=1$.
Assume that $q>1$.
We use the induction hypothesis on $V_1$.
Since $(C_1l_1)_{P_1}=(q-1)m_1+r$,
we have
$(C_q l_q)_{P_q}=r$, $l_q^2=l_1^2-q+1=l^2-q$,
$\tau_{q,1}^{-1}(l_1)=E_2+\cdots+E_q+l_q$,
$E_q l_q=1$ and $E_2 l_q=\cdots=E_{q-1} l_q=0$.
Since $E_1l_1=1$ on $V_1$, the curve $E_1$ does not intersect $l_q$.
\end{proof}
\begin{proof}[Proof of Proposition~\ref{prop:ms}]
We first show the assertion for $A_1$ and $B_1$ by induction on $n_1$.
Put $b_i=1+\sum_{j=1}^{i}a_{1,j}$.
By applying Lemma~\ref{lem:ch1} to $(C,P)$ with $q=b_1$,
we have $\tau_{b_1,0}^{-1}(P)=[E_1,\ldots,E_{b_1-1},E_{b_1}]=[\TW{a_{1,1}},1]$ and
$(C_{b_1} E_{b_1})_{P_{b_1}}=m_{1,1}$.
We see $\ms_{P_{b_1}}(C_{b_1})=((m_{1,2})_{a_{1,2}},\ldots)$.
\begin{center}
\includegraphics{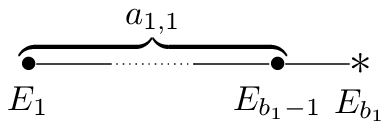}
\end{center}
We next apply Lemma~\ref{lem:ch1} to $(C_{b_1},P_{b_1})$ with $q=a_{1,2}$.
We have $\tau_{b_2,b_1}^{-1}(P_{b_1})=[E_{b_1+1},\ldots,E_{b_2-1},E_{b_2}]=[\TW{a_{1,2}},1]$ and $(C_{b_2} E_{b_2})_{P_{b_2}}=m_{1,2}$.
\begin{center}
\includegraphics{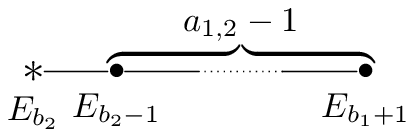}
\end{center}

We then apply Lemma~\ref{lem:ch2} to $E_{b_1}$ and $(C_{b_1},P_{b_1})$.
Because
$(C_{b_1}E_{b_1})_{P_{b_1}}=a_{1,2}m_{1,2}$ ($n_1=2$)
or
$(C_{b_1}E_{b_1})_{P_{b_1}}=a_{1,2}m_{1,2}+m_{1,3}$ ($n_1>2$),
it follows that
$\tau_{b_2,b_1}^{-1}(E_{b_1})=[E_{b_1+1},\ldots,E_{b_2-1},E_{b_2},E_{b_1}]=[\TW{a_{1,2}-1},1,1+a_{1,2}]$
and that $(C_{b_2}E_{b_1})_{P_{b_2}}=0$ ($n_1=2$) or 
$(C_{b_2}E_{b_1})_{P_{b_2}}=m_{1,3}$ ($n_1>2$).
Since $P_{b_1}\not\in E_i$ for $i<b_1$,
we see
$\tau_{b_2,0}^{-1}(P)=[E_1,\ldots,E_{b_1-1},E_{b_1},E_{b_2},E_{b_2-1},\ldots,E_{b_1+1}]=[\TW{a_{1,1}},1+a_{1,2},1,\TW{a_{1,2}-1}]$.
\begin{center}
\includegraphics{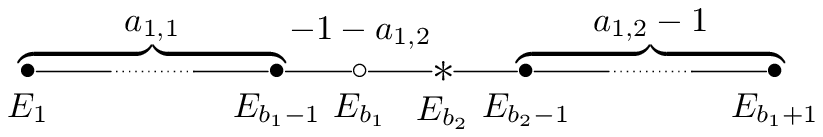}
\end{center}

Suppose that $n_1=2$.
Since $m_{1,1}=a_{1,2}m_{1,2}$,
we have $P_{b_2}\not\in E_{i}$ for $i< b_2$.
Thus the weighted dual graph of $\tau_{b_2,0}^{-1}(P)-E_{b_2}$
is unchanged by the remaining blow-ups.
The vertex corresponding to $E_{b_2}$ is a branching component of
the dual graph of $\sigma^{-1}(P)+C'$.
Because $A_1$ contains $E_1$,
we have $A_1=\TW{a_{1,1}+1}\TA[a_{1,2}]$, $B_1=\TW{a_{1,2}-1}$.

Suppose that $n_1>2$.
We have $\ms_{P_{b_1}}(C_{b_1})=((m_{1,2})_{a_{1,2}},m_{1,3},\ldots)$.
Since $(C_{b_2}E_{b_1})_{P_{b_2}}=m_{1,3}=\mult_{P_{b_2}}(C_{b_2})$,
we see $P_{b_2}\in E_{b_1}$ and $P_i\not\in E_{b_1}$ for $i>b_2$.
It follows that
$E_{b_1}^2=-a_{1,2}-2$ and that $E_{b_1}$ intersects $E_{b_2+1}$
on $V_i$ for $i>b_2$.
We apply the induction hypothesis
to $(C_{b_1},P_{b_1})$.
Put $T=\tau_{b_{n_1},b_1}^{-1}(P_{b_1})$.
We write it as $T=[A,1,B]$, where $A$ contains $E_{b_1+1}$.
If $n_1$ is an odd number, then
\begin{eqnarray*}
&A=\TW{a_{1,2}}\TA[a_{1,3}]\TA\TW{a_{1,4}+1}\TA[a_{1,5}]\TA\cdots\TA\TW{a_{1,n_1-1}+1}\TA[a_{1,n_1}],\\
&B=\TW{a_{1,n_1}}\TA[a_{1,n_1-1}]\TA\TW{a_{1,n_1-2}+1}\TA\cdots\TA[a_{1,6}]\TA\TW{a_{1,5}+1}\TA[a_{1,4}]\TA\TW{a_{1,3}}.
\end{eqnarray*}
If $n_1$ is an even number, then
\begin{eqnarray*}
&A=\TW{a_{1,2}}\TA[a_{1,3}]\TA\TW{a_{1,4}+1}\TA[a_{1,5}]\TA\cdots\TA\TW{a_{1,n_1-2}+1}\TA[a_{1,n_1-1}]\TA\TW{a_{1,n_1}},\\
&B=[a_{1,n_1}]\TA\TW{a_{1,n_1-1}+1}\TA\cdots\TA[a_{1,6}]\TA\TW{a_{1,5}+1}\TA[a_{1,4}]\TA\TW{a_{1,3}}.
\end{eqnarray*}

The first curve of $A$ is $E_{b_1+1}$ by Lemma~\ref{lem:cres} (iii).
It follows that
$\tau_{b_{n_1},0}^{-1}(P)=[E_{1},\ldots,E_{b_1-1},E_{b_1},\TP{T}]$.
By the induction hypothesis,
$A$ and $B$ are unchanged by the remaining blow-ups.
We infer that
$\tau_{b_{n_1},0}^{-1}(P)-E_{b_{n_1}}$ is also
unchanged by the remaining blow-ups.
Hence $A_1=[\TW{a_{1,1}},a_{1,2}+2,\TP{B}]$, $B_1=\TP{A}$.
We can prove the assertion for $A_i$ and $B_i$ with $i\ge 2$
by using the same arguments as above, where
$(C_b,P_b)$ ($b=\sum_{j=1}^{i-1}\sum_{k=1}^{n_j}a_{j,k}$)
plays the role of $(C,P)$.
\end{proof}
%
%
%%%%%%%%%%%%%%%%%%%%%%%%%%%%%%%%%%%%%%%%%%%%%%%%%%%%%%%%%%%%%%%%%%%%%%%%%%%%%%%
\section{%
Proof of Theorem~\ref{thm0}}\label{sec:bck1}
%%%%%%%%%%%%%%%%%%%%%%%%%%%%%%%%%%%%%%%%%%%%%%%%%%%%%%%%%%%%%%%%%%%%%%%%%%%%%%%
Let $C$ be a rational bicuspidal plane curve.
Let $P_1,P_2$ denote the cusps of $C$.
Let $\sigma:V\rightarrow\SP^2$ be the minimal embedded resolution
of the cusps
and $C'$ the strict transform of $C$ via $\sigma$.
Put $D:=\sigma^{-1}(C)$.
We may assume $\sigma=\sigma^{(1)}\circ\sigma^{(2)}$, where
$\sigma^{(k)}$ consists of the blow-ups over $P_k$.
We decompose
the dual graph of $\sigma^{-1}(P_k)$ ($k=1,2$)
into subgraphs
$A^{(k)}_1,B^{(k)}_1,\ldots,A^{(k)}_{g_k},B^{(k)}_{g_k},D^{(k)}_{0}$
in the same way as in Lemma~\ref{lem:cres}.
\begin{center}
\includegraphics{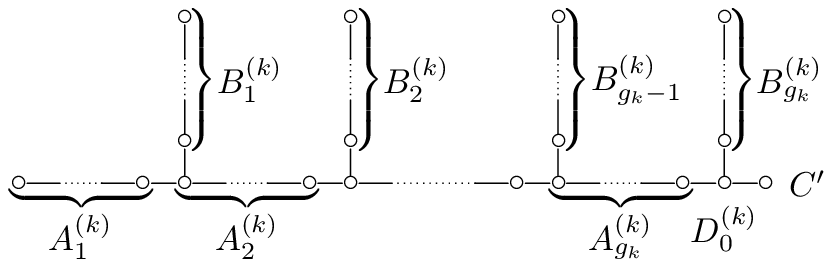}
\end{center}
By definition,
$A^{(k)}_1$ contains the exceptional curve of the first blow-up
over $P_k$.
We give the weighted graphs $A^{(k)}_1,\ldots,A^{(k)}_{g_k}$
(resp.~$B^{(k)}_1,\ldots,B^{(k)}_{g_k}$)
the direction
from the left-hand side to the right
(resp.~from the bottom to the top) of the above figure.
With these directions,
we regard $A^{(k)}_i$ and $B^{(k)}_i$ as linear chains.
Let $\sigma^{(k)}_0$ denote the first blow-up of $\sigma^{(k)}$.
By Lemma~\ref{lem:cres},
there exists a decomposition
$\sigma^{(k)}=\sigma^{(k)}_0\circ\sigma^{(k)}_{1,1}\circ\sigma^{(k)}_{1,2}\circ\cdots\circ\sigma^{(k)}_{g_k,1}\circ\sigma^{(k)}_{g_k,2}$
such that
each $\sigma^{(k)}_{i,1}$ (resp.~$\sigma^{(k)}_{i,2}$) consists of
sprouting (resp.~subdivisional) blow-ups
with respect to preimages of $P_k$.
The morphism $\sigma^{(k)}_{i,1}\circ\sigma^{(k)}_{i,2}$
contracts $[A^{(k)}_i,1,B^{(k)}_i]$ to a ($-1$)-curve
for $i\ge1$.
Let $\NS^{(k)}_i$ denote the number of the blow-ups
of $\sigma^{(k)}_{i,1}$.
We first show the ``if'' part of Theorem~\ref{thm0}.
Assume that $\KB(\SP^2\setminus C)=1$.
Put $D_{1}^{(k)}=B_{g_k}^{(k)}$ and
$D_{2}^{(k)}=D^{(k)}-(D_{0}^{(k)}+D_{1}^{(k)})$.
The dual graph of $D$ has the following shape.
\begin{center}
\includegraphics{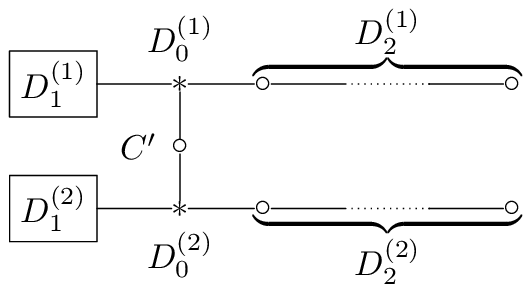}
\end{center}
Following \cite{fz:def},
we consider a strictly minimal model
$(\widetilde{V},\widetilde{D})$ of $(V,D)$.
We successively contract ($-1$)-curves $E$ satisfying
one of the following conditions:
(1) $E \subset D$ and $(D-E) E=0$,
(2) $E \subset D$ and $(D-E) E=1$,
(3) $E \subset D$ and $(D-E) E=2$,
(4) $E \not\subset D$ and $D E=0$,
(5) $E \not\subset D$ and $D E=1$.
After a finite number of contractions,
we have no $(-1)$-curves satisfying the above conditions.
Let $\pi:V \rightarrow \widetilde{V}$ be the composite of the contractions.
\begin{lemma}\label{lem:pi}
The morphism $\pi$ does not contract irreducible curves meeting with $C'$.
In particular,
$(C')^2=-1$ if and only if $C'$ is contracted by $\pi$.
\end{lemma}
\begin{proof}
Suppose that there exists an irreducible curve $E$ on $V$
which intersects $C'$ and is contracted by $\pi$.
If $E$ is a component of $D$,
then $E$ is either $D_{0}^{(1)}$ or $D_{0}^{(2)}$.
Since $E$ is a ($-1$)-curve, we may assume that
$\pi$ contracts $E$ first.
But this contraction is not allowed, since $(D-E) E=3$.
Thus  $E\not\subset D$.
Since $E$ is contracted by $\pi$,
$E$ does not intersect any components of $D$ other than $C'$.
This means that $\sigma(E)$ is a plane curve with $\sigma(E)^2 \le -1$,
which is impossible.
\end{proof}

For a divisor $E$ on $V$, we write $\widetilde{E}=\pi_{\ast}(E)$.
It is clear that $\widetilde{D}$ is an SNC-divisor
and $\KB(\widetilde{V}\setminus\widetilde{D})=1$.
\begin{lemma}
There exists a fibration $\widetilde{p}:\widetilde{V}\rightarrow \SP^1$
whose general fiber $F$ is $\SP^1$ and $\widetilde{D} F=2$.
\end{lemma}
\begin{proof}
By \cite[Theorem 2.3]{ka:cls}
and the fact that $\widetilde{V}\setminus \widetilde{D}$ is affine,
there exists a fibration $\widetilde{p}:\widetilde{V}\rightarrow W$
over a smooth curve $W$
whose general fiber $F$ is $\SP^1$ and $\widetilde{D} F=2$.
Since $q(\widetilde{V})=0$, the curve $W$ must be $\SP^1$.
\end{proof}

The fibration $\widetilde{p}$ is obtained from
a $\SP^1$-bundle $\hat{p}:\Sigma\rightarrow\SP^1$
by successive blow-ups $\widetilde{\pi}:\widetilde{V}\rightarrow\Sigma$.
Putting $p=\widetilde{p}\circ\pi$,
we have the following commutative diagram.
\begin{center}
\includegraphics{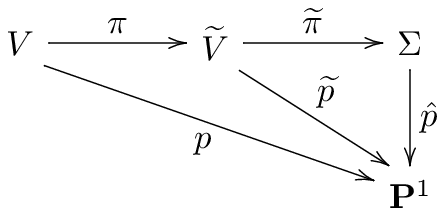}
\end{center}

Following \cite{fz:def}, we use the following terminology.
The triple $(\widetilde{V},\widetilde{D},\widetilde{p})$
is called a \emph{$\SC^{\ast}$-triple}.
A component of $\widetilde{D}$ is called \emph{horizontal} if
the image of it under $\widetilde{p}$ is 1-dimensional.
Let $\widetilde{H}$ be
the sum of the horizontal components of $(\widetilde{V},\widetilde{D},\widetilde{p})$.
The $\SC^{\ast}$-triple $(\widetilde{V},\widetilde{D},\widetilde{p})$
is called of \emph{twisted type}
if $\widetilde{H}$ is irreducible;
otherwise it is called of \emph{untwisted type}.
A fiber of $\widetilde{p}$ is called a \emph{full fiber}
of $(\widetilde{V},\widetilde{D},\widetilde{p})$ if it is contained in $\widetilde{D}$.
Let $f$ denote
the number of the full fibers of $(\widetilde{V},\widetilde{D},\widetilde{p})$.
\begin{lemma}\label{lem:a1}
The $\SC^{\ast}$-triple has the following properties.
\begin{enumerate}
\item[\textnormal{(i)}]
The $\SC^{\ast}$-triple is of untwisted type.
\item[\textnormal{(ii)}]
We have $f\le 1$.
The fibration $\widetilde{p}$ has at least two singular fibers.
\item[\textnormal{(iii)}]
The weighted dual graph of a singular fiber of $\widetilde{p}$ is a linear chain $[A,1,B]$,
where $A$, $B$ are admissible and
are connected components of $\widetilde{D}-\widetilde{H}$.
The curve $\widetilde{H}$
intersects only the first vertex of $A$ and the last of $B$.
\end{enumerate}
\end{lemma}
\begin{proof}
By \cite[Theorem 3]{kiz}, the $\SC^{\ast}$-triple is of untwisted type.
The assertions (ii), (iii)
follow from \cite[Lemma 4.4, Theorem 5.8 and 5.11]{fz:def}.
\end{proof}

Lemma~\ref{lem:pi} and
the assertion (i) of the following proposition
show the ``if'' part of Theorem~\ref{thm0}.
\begin{proposition}\label{prop:unt}
The following assertions hold.
\begin{enumerate}
\item[\textnormal{(i)}]
We have $\widetilde{H}=\widetilde{D}_{0}^{(1)}+\widetilde{D}_{0}^{(2)}$.
The curve $\widetilde{C}'$ is a full fiber of $\widetilde{p}$.
\item[\textnormal{(ii)}]
The fibration $\widetilde{p}$ has exactly two singular fibers
$\widetilde{F}_1=\widetilde{D}_{1}^{(1)}+\widetilde{E}_1+\widetilde{D}_{a}^{(2)}$,
$\widetilde{F}_2=\widetilde{D}_{2}^{(1)}+\widetilde{E}_2+\widetilde{D}_{b}^{(2)}$,
where $\{a,b\}=\{1,2\}$ and $\widetilde{E}_i$ is the ($-1$)-curve in $\widetilde{F}_i$.
\end{enumerate}
\end{proposition}
\begin{proof}
We first show that $\pi$ does not contract $C'$.
Assume the contrary.
Since $(\widetilde{D}_{0}^{(1)})^2 \ge 0$,
$\widetilde{D}_{0}^{(1)}$ is either a horizontal component or a full fiber.
Assume $\widetilde{D}_{0}^{(1)}$ is a full fiber.
Since $(\widetilde{D}-\widetilde{D}_{0}^{(1)})\widetilde{D}_{0}^{(1)}=2$,
one of $D_{1}^{(1)}$ or $D_{2}^{(1)}$ is contracted by $\pi$ to
a point on $\widetilde{D}_{0}^{(1)}$.
Thus we have $(\widetilde{D}_{0}^{(1)})^2>0$, which is a contradiction.
Similarly, $\widetilde{D}_{0}^{(2)}$ is not a full fiber.
Thus $\widetilde{H}=\widetilde{D}_{0}^{(1)}+\widetilde{D}_{0}^{(2)}$.
Let $\widetilde{F}$ be the fiber of $\widetilde{p}$ passing through
the point of intersection of 
$\widetilde{D}_{0}^{(1)}$ and $\widetilde{D}_{0}^{(2)}$.
The strict transform $F$ of $\widetilde{F}$ in $V$
intersects only $C'$ among the irreducible components of $D$.
Hence $\sigma(F)$ is a plane curve with $\sigma(F)^2<0$,
which is a contradiction.
Thus $\pi$ does not contract $C'$.

Since $(\widetilde{D}_{0}^{(1)})^2 \ge -1$,
$\widetilde{D}_{0}^{(1)}$ is either a horizontal component or a full fiber.
Suppose that $\widetilde{D}_{0}^{(1)}$ is a full fiber.
Then $\widetilde{C}'$ must be a horizontal component.
This means that $\widetilde{p}$
has at most one singular fiber, which contradicts Lemma~\ref{lem:a1}.
Thus $\widetilde{D}_{0}^{(1)}$ is a horizontal component.
Similarly, $\widetilde{D}_{0}^{(2)}$ is a horizontal component.
Hence $\widetilde{C}'$ must be a full fiber of $\widetilde{p}$.
The assertion (ii) follows from (i) and Lemma~\ref{lem:a1}.
\end{proof}

We prove the remaining assertions of Theorem~\ref{thm0}.
Let $C$ be a rational bicuspidal plane curve.
Suppose $(C')^2\ge 0$.
Since $\dim |C'|=1+(C')^2$,
it follows that $\SP^2\setminus C$ contains
a surface $\SC^{\ast}\times B$, where $B$ is a curve.
Hence we have $\KB(\SP^2\setminus C)\le 1$.
By \cite{Wak}, $\KB(\SP^2\setminus C)\ge 0$.
By \cite[Proposition 1]{Ts},
$\KB(\SP^2\setminus C)\ge 1$.
See also \cite{ko,or}.
Hence we have $\KB(\SP^2\setminus C)=1$ and $(C')^2=0$.
%
%
%%%%%%%%%%%%%%%%%%%%%%%%%%%%%%%%%%%%%%%%%%%%%%%%%%%%%%%%%%%%%%%%%%%%%%%%%%%%%%%
\section{Proof of Theorem~\ref{thm1}}\label{sec:str}
%%%%%%%%%%%%%%%%%%%%%%%%%%%%%%%%%%%%%%%%%%%%%%%%%%%%%%%%%%%%%%%%%%%%%%%%%%%%%%%
Let $C$ be a rational bicuspidal plane curve.
Let $P_1,P_2$ denote the cusps of $C$.
Let $\sigma:V\rightarrow\SP^2$ be the minimal embedded resolution
of the cusps.
Let $C'$, $D$, etc.~have the same meaning as in
the first paragraph of the previous section.
Assume that $(C')^2=-1$.
Put $F_0'=D^{(1)}_0$.
Let $\sigma':V\rightarrow V'$ be the contraction of $C'$.
Since $(F_0')^2=0$ on $V'$,
there exists a $\SP^1$-fibration $p':V'\rightarrow\SP^1$
such that $F_0'$ is a nonsingular fiber.
Put $p=p'\circ\sigma':V\rightarrow\SP^1$
and $F_0=F_0'+C'$.
\begin{remark*}
Since $(D^{(2)}_0)^2=0$ on $V'$,
there exists another $\SP^1$-fibration
such that $D^{(2)}_0$ is a nonsingular fiber.
\end{remark*}
The surface $X=V\setminus D$ is a $\SQ$-homology plane.
Namely $h^i(X,\SQ)=0$ for $i>0$.
A general fiber of $p|_{X}$ is a curve
$\SC^{\ast\ast}=\SP^1\setminus\{3\text{ points}\}$.
Such fibrations have already been classified in \cite{misu}.
We will use their result to prove our theorem.
There exists a birational morphism $\varphi:V\rightarrow\Sigma_n$
from $V$ onto the Hirzebruch surface $\Sigma_n$ of degree $n$ for some $n$
such that $p\circ\varphi^{-1}:\Sigma_n\rightarrow\SP^1$ is a $\SP^1$-bundle.
The morphism $\varphi$ is the composite of the successive contractions
of the ($-1$)-curves in the singular fibers of $p$.
Let $S_1$ and $S_3$ be the irreducible components of $A^{(1)}_{g_1}+B^{(1)}_{g_1}$
meeting with $D^{(1)}_0$.
Put $S_2=D^{(2)}_0$.
The curves $S_1$, $S_2$ and $S_3$ are 1-sections of $p$.
The divisor $D$ contains no other sections of $p$.
\begin{lemma}\label{lem:sec}
We may assume that $\varphi(S_1+S_2+S_3)$ is smooth.
We have
$\varphi(S_1)\sim \varphi(S_2)\sim \varphi(S_3)$ (linearly equivalent) and
$n=\varphi(S_i)^2=0$ for each $i$.
\end{lemma}
\begin{proof}
We only prove the first assertion.
Suppose $\varphi(S_1+S_2+S_3)$ has a singular point $P$.
Let $\phi_1$ be the blow-up at $P$.
Since $S_1+S_2+S_3$ is smooth on $V$,
we can arrange the order of the blow-ups of $\varphi$
so that $\varphi=\phi_1\circ\varphi'$.
Let $F'$ be the strict transform via $\phi_1$
of the fiber of $p\circ\varphi^{-1}$ passing through $P$.
Let $\phi_2$ be the contraction of $F'$.
Since $F'$ is an irreducible component of
a singular fiber of $p\circ{\varphi'}^{-1}$,
we can replace $\varphi$ with $\phi_2\circ\varphi'$.
We infer that $P$ can be resolved by repeating the above process.
Hence we may assume that $\varphi(S_1+S_2+S_3)$ is smooth.
\end{proof}

Let $F_1,\ldots,F_l$ be all singular fibers of $p$ other than $F_0$.
For $i=1,\ldots,l$,
let $E_i$ be the sum of the irreducible components of $F_i$
which are not components of $D$.
Since $D$ contains no loop, each $E_i$ is not empty.
It follows that
the base curve of the $\SC^{\ast\ast}$-fibration $p|_X$ is $\SC$.
Because $\KB(V\setminus D)=2$,
each irreducible component of $E_i$ meets with $D$ in at least two points
by \cite[Main Theorem]{mits2}.
In \cite[Lemma 1.5]{misu},
singular fibers of a $\SC^{\ast\ast}$-fibration
with three 1-sections
were classified into several types.
Among them,
only singular fibers
of type  ($\mathrm{I_1}$) and ($\mathrm{III_1}$)
satisfy the conditions that
each irreducible component of $E_i$ meets with $D$ in at least two points.
From the fact that $D$ contains no loop,
we infer that
each $F_i$ is of type ($\mathrm{III_1}$).
By \cite[Lemma 2.3]{misu},
$p$ has at most two singular fibers other than $F_0$.
Since $S_2$ meets with $D-S_2$ in three points,
$p$ has exactly three singular fibers $F_0$, $F_1$ and $F_2$.
For $i=1,2$,
the dual graph of $F_i+S_1+S_2+S_3$ coincides with
one of those in Figure~\ref{figsf},
where $\ast$ denotes a ($-1$)-curve and $E_i=E_{i1}+E_{i2}$.
The graph $T_{i,j}$ may be empty for each $j$.
\begin{figure}
\begin{center}
\includegraphics{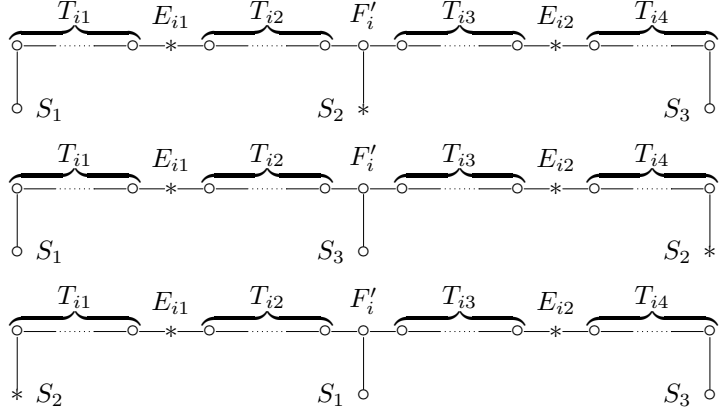}
\end{center}
\caption{Candidates for the dual graph of $F_i$, $i=1,2$}\label{figsf}
\end{figure}
\begin{lemma}\label{lem:fib}
We have $\varphi(F_i)=\varphi(F_i')$ for $i=0,1,2$.
For $i=1,2$,
the dual graph of $F_i+S_1+S_2+S_3$ must be
the first one in Figure~\ref{figsf}.
\end{lemma}
\begin{proof}
By Lemma~\ref{lem:sec},
we have $\varphi(F_0)=\varphi(F_0')$.
Suppose that $\varphi$ contracts $F_1'$.
Let $F_1'$ intersect $S_j$.
Write $\varphi$ as $\varphi=\varphi_3\circ\varphi_2\circ\varphi_1$,
where $\varphi_2$ is the contraction of $F_1'$.
The curve $\varphi_1(F_1')$ intersects three irreducible components of $D+E_1+E_2$.
By Lemma~\ref{lem:sec},
$\varphi_1(F_1')$ does not intersect the images under $\varphi_1$
of sections other than $S_j$.
It follows that
$\varphi_{2}(\varphi_1(S_j))\varphi_{2}(\varphi_1(F_1))>1$, which is absurd.
Thus $\varphi$ does not contract $F_1'$.
Similarly, $\varphi$ does not contract $F_2'$.
If one of $F_1'$, $F_2'$ does not intersect $S_2$,
then $\varphi(S_2)^2> 0$, which contradicts Lemma~\ref{lem:sec}.
Thus $F_1'$ and $F_2'$ intersect $S_2$.
\end{proof}

By Lemma~\ref{lem:fib},
the dual graph of $D+E_1+E_2$ must coincide with that in Figure~\ref{fig}.
%%%%%%%%%%%%%%%%%%%%%%%%%%%%%%%%%%%%%%%%%%%%%%%%%%%%%%%%%%%%%%%%%%%%%%%%%%%%%%%
\section{Proof of Theorem~\ref{thm1} --- continued}
%%%%%%%%%%%%%%%%%%%%%%%%%%%%%%%%%%%%%%%%%%%%%%%%%%%%%%%%%%%%%%%%%%%%%%%%%%%%%%%
%
Let the notation be as in the previous section.
We infer $g_i\le 2$ for $i=1,2$.
With the direction from the left-hand side to the right of Figure~\ref{fig},
we regard $T_{ij}$'s as linear chains.
Put $s_i=-S_i^2$ and $f_j=-(F_j')^2$ for $i\ne 2$ and $j=1,2$.
We have $s_i\ge 2$ and $f_j\ge 2$.
\begin{lemma}\label{lem:0}
The following assertions hold.
\begin{enumerate}
\item[(i)]
We may assume $B^{(1)}_{g_1}=[S_1,T_{11}]$.
We have $T_{21}=\emptyset$.
There exists a non-negative integer $l_{22}$ such that $T_{22}=\TW{l_{22}}$.
\item[(ii)]
There exist positive integers $k_{12}$, $k_{34}$ such that
$[S_1,T_{11}]^{\ast}=[T_{12},k_{12}+1,\TW{l_{22}}]$,
$[F_1',T_{13}]^{\ast}=[T_{14},k_{34}+1,\TW{k_{12}-1}]$ and
$[T_{24},S_3]^{\ast}=[\TW{k_{34}-1},f_2,T_{23}]$.
We have
$A^{(1)}_{g_1}=\TW{\NS^{(1)}_{g_1}}\TA[T_{12},k_{12}+1,\TW{l_{22}}]$.
\end{enumerate}
\end{lemma}
\begin{proof}
(i)
We may assume $B^{(1)}_{g_1}=[S_1,T_{11}]$
because
the dual graph of $D+E_1+E_2$ is symmetric
about the line passing through $F_1'$, $S_2$ and $F_2'$ in Figure~\ref{fig},
and the line passing through $S_1$, $S_2$ and $S_3$.
We have $T_{21}=\emptyset$.
If $T_{22}\ne\emptyset$,
then $\varphi$ contracts $[E_{21},T_{22}]$ to a ($-1$)-curve.
By Lemma~\ref{lem:bu},
there exists a positive integer $l_{22}$ such that $T_{22}=\TW{l_{22}}$.
We set $l_{22}=0$ if $T_{22}=\emptyset$.

(ii)
We may assume that
$\varphi=\varphi_0\circ\varphi_{21}\circ\varphi_{11}\circ\varphi_{12}\circ\varphi_{22}$, where $\varphi_0$ contracts $C'$ and $\varphi_{ij}$ contracts
$T_{i,2j-1}+E_{ij}+T_{i,2j}$ to a point.
Since $[E_{21},T_{22}]=[1,\TW{l_{22}}]$ and $\varphi(S_1)^2=0$,
$\varphi_{11}$ contracts $[S_1,T_{11},E_{11},T_{12}]$ to $[l_{22}+1]$
by Lemma~\ref{lem:indf} (iii).
By Lemma~\ref{lem:bu},
there exists a positive integer $k_{12}$ such that
$[S_1,T_{11}]^{\ast}=[T_{12},k_{12}+1,\TW{l_{22}}]$.
The composite of the subdivisional blow-ups of $\varphi_{11}$
with respect to the preimages of $S_1$ contracts
$[T_{11},E_{11},T_{12}]$ to $[\TW{k_{12}-1},1]$.
Since $\varphi(F_1')^2=0$,
$\varphi_{12}$ contracts $[F_1',T_{13},E_{12},T_{14}]$ to $[k_{12}]$
by Lemma~\ref{lem:indf} (iii).
By Lemma~\ref{lem:bu},
there exists a positive integer $k_{34}$ such that
$[F_1',T_{13}]^{\ast}=[T_{14},k_{34}+1,\TW{k_{12}-1}]$.
Similarly,
$\varphi_{22}$ contracts $[T_{23},E_{22},T_{24},S_3]$ to $[k_{34}]$.
By Lemma~\ref{lem:bu},
there exists a positive integer $k$ such that
$[S_3,\TP{T}_{24}]^{\ast}=[\TP{T}_{23},k+1,\TW{k_{34}-1}]$.
Since $\varphi(F_2')^2=0$,
we have $0=-f_2+k+1$.
The last assertion follows from
Proposition~\ref{prop:cres}.
\end{proof}

Now we prove Theorem~\ref{thm1}.
The linear chain $B^{(2)}_{g_2}$ coincides with one of
$\TP{[T_{12},F_1']}$,
$[F_1',T_{13}]$,
$\TP{[T_{22},F_2']}$ or
$[F_2',T_{23}]$.
\begin{figure}
\begin{center}
\includegraphics{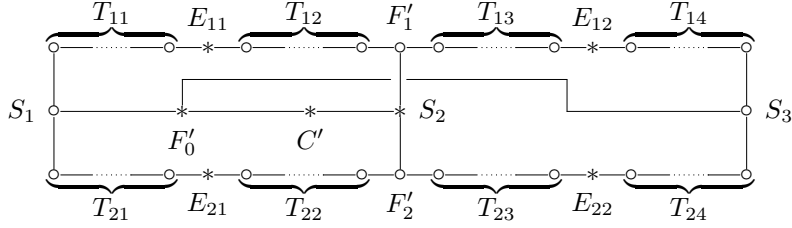}
\end{center}
\caption{The dual graph of $D+E_1+E_2$}\label{fig}
\end{figure}
%
%%%%%%%%%%%%%%%%%%%%%%%%%%%%%%%%%%%%%%%%%%%%%%%%%%%%%%%%%%%%%%%%%%%%%%%%%%%%%%%
\subsection{$B^{(2)}_{g_2}=\protect\TP{[T_{12},F_1']}$}\label{subsec:1}
\begin{lemma}\label{lem:1}
We have $T_{13}=\emptyset$, $k_{34}=1$, $T_{14}=\TW{f_1-k_{12}-1}$ and
$[T_{24},S_3]^{\ast}=[f_2,T_{23}]$.
\end{lemma}
\begin{proof}
It is clear that $T_{13}=\emptyset$.
By Lemma~\ref{lem:0}, we see
$[T_{14},k_{34}+1,\TW{k_{12}-1}]=\TW{f_1-1}$.
This means that $k_{34}=1$, $T_{14}=\TW{f_1-k_{12}-1}$.
By Lemma~\ref{lem:0}, we have
$[T_{24},S_3]^{\ast}=[f_2,T_{23}]$.
\end{proof}

Case (1): $g_2=1$.
Either $A^{(2)}_1=[T_{22},F_2']$ or $A^{(2)}_1=\TP{[F_2',T_{23}]}$.
Suppose $A^{(2)}_1=[T_{22},F_2']$.
We have $T_{23}=\emptyset$.
By Lemma~\ref{lem:1},
we get $[T_{24},S_3]=\TW{f_2-1}$.
Thus $T_{14}+S_3+T_{24}$ consists of ($-2$)-curves
and so does $A^{(1)}_{g_1}$, which contradicts Proposition~\ref{prop:cres}.
Hence $A^{(2)}_1=\TP{[F_2',T_{23}]}$.
It follows that $T_{22}=\emptyset$.
By Lemma~\ref{lem:1} and Proposition~\ref{prop:cres},
we obtain $[T_{24},S_3]=[\NS^{(2)}_1+1,T_{12},F_1']$.
We have $T_{24}\ne\emptyset$.

Suppose $T_{14}=\emptyset$.
We have $f_1=k_{12}+1$,
$g_1=1$ and $A^{(1)}_1=[T_{24},S_3]$.
Since $\TW{l_{22}}=T_{22}=\emptyset$,
we get $B^{(1)}_1=[T_{12},f_1]^{\ast}$ by Lemma~\ref{lem:0}.
By Proposition~\ref{prop:cres},
$[\NS^{(2)}_1+1,T_{12},f_1]=[T_{24},S_3]=\TW{\NS^{(1)}_1}\TA[T_{12},f_1]$.
Thus $\NS^{(1)}_1=2$, $\NS^{(2)}_1=1$.
Since $[T_{12},f_1]=\TW{1}\TA[T_{12},f_1]$,
we infer $[T_{12},f_1]^{\ast}=[[T_{12},f_1]^{\ast},2]$, which is absurd.
Hence $T_{14}\ne\emptyset$.
If $g_1=1$, then $A^{(1)}_1=[T_{14},s_3]$ and $T_{24}=\emptyset$,
which is a contradiction.
Thus $g_1=2$.
Since $A^{(1)}_2=S_3$,
we have $T_{12}=\emptyset$, $\NS^{(1)}_2=1$ and $A^{(1)}_2=S_3=[k_{12}+2]$
by Lemma~\ref{lem:0}.
By Proposition~\ref{prop:cres},
$B^{(1)}_2=\TW{k_{12}}$.

By Lemma~\ref{lem:1},
$T_{14}$ consists of ($-2$)-curves.
By Proposition~\ref{prop:cres},
$B^{(1)}_1=\TP{T}_{14}$ and $A^{(1)}_1=T_{24}$.
By Lemma~\ref{lem:1} and Proposition~\ref{prop:cres},
we get
$[f_2,T_{23}]=\TW{k_{12}+1}\TA[\TW{f_1-k_{12}-1},\NS^{(1)}_1+1]$.
Since $\emptyset\ne T_{14}=\TW{f_1-k_{12}-1}$,
we have $[f_2,T_{23}]=[\TW{k_{12}},3,\TW{f_1-k_{12}-2},\NS^{(1)}_1+1]$.
On the other hand,
$[f_2,T_{23}]=\TP{A}^{(2)}_1=\TW{f_1-1}\TA\TW{\NS^{(2)}_1}=[\TW{f_1-2},3,\TW{\NS^{(2)}_1-1}]$ by Proposition~\ref{prop:cres}.
This means that $\NS^{(2)}_1=2$, $\NS^{(1)}_1=1$ and $f_1=k_{12}+2$.
Write $k=k_{12}$.
We have
$A^{(1)}_2=[k+2]$,
$B^{(1)}_2=\TW{k}$,
$B^{(1)}_1=[2]$ and
$B^{(2)}_1=[k+2]$.
By Proposition~\ref{prop:cres},
we see $A^{(1)}_1=[3]$ and
$A^{(2)}_1=[2,3,\TW{k}]$.
It follows from Proposition~\ref{prop:ms} that the numerical data of
$C$ is equal to
$\{(2(k+1),(k+1)_2),\;\;((k+2)_2,k+1)\}$,
which coincides with the data 2 with $a=1$, $b=k+1$.

Case (2): $g_2=2$.
By Lemma~\ref{lem:0}, $T_{22}$ consists of ($-2$)-curves.
By Proposition~\ref{prop:cres},
we have $B^{(2)}_1=\TP{T}_{22}$ and $A^{(2)}_1=\TP{T}_{23}$.
Thus $A^{(2)}_1=\TW{\NS^{(2)}_1}\TA[l_{22}+1]$.
Since $T_{22}\ne\emptyset$,
we infer
$\NV(A^{(1)}_{g_1})\ge2$ by Lemma~\ref{lem:0}.
It follows that $g_1=1$.
Either $A^{(1)}_1=[T_{14},S_3]$ or $A^{(1)}_1=[T_{24},S_3]$.
By Lemma~\ref{lem:1},
$T_{14}$ consists of ($-2$)-curves.
By Lemma~\ref{lem:0},
$A^{(1)}_1-S_3$ contains an irreducible component other than a ($-2$)-curve.
Thus $A^{(1)}_1=[T_{24},S_3]$, $T_{14}=\emptyset$ and $f_1=k_{12}+1$.

Because $A^{(2)}_2=F_2'$,
we have $\TW{f_2-1}=[\NS^{(2)}_2+1,T_{12},f_1]$ by Proposition~\ref{prop:cres}.
This shows that $\NS^{(2)}_2=1$, $f_1=2$ and $T_{12}=\TW{f_2-3}$.
Thus $k_{12}=1$.
By Lemma~\ref{lem:0},
$[S_1,T_{11}]=\TW{f_2+l_{22}-2}^{\ast}=[f_2+l_{22}-1]$.
By Proposition~\ref{prop:cres} and Lemma~\ref{lem:1},
$[f_2,T_{23}]={A^{(1)}_1}^{\ast}=[S_1,T_{11},\NS^{(1)}_1+1]=[f_2+l_{22}-1,\NS^{(1)}_1+1]$.
Hence $l_{22}=1$, $T_{23}=[\NS^{(1)}_1+1]$.
We infer $\NS^{(2)}_1=1$ and $\NS^{(1)}_1=2$.
Put $k=f_2-2$.
We have $k\ge 1$,
$B^{(1)}_1=[k+2]$,
$A^{(2)}_2=[k+2]$,
$B^{(2)}_1=[2]$ and
$A^{(2)}_1=[3]$.
By Proposition~\ref{prop:cres},
we get
$A^{(1)}_1=[2,3,\TW{k}]$ and
$B^{(2)}_2=\TW{k}$.
The numerical data of
$C$ is equal to
$\{(2(k+1),(k+1)_2),\;\;((k+2)_2,k+1)\}$,
which coincides with the data 2 with $a=1$, $b=k+1$.
%%%%%%%%%%%%%%%%%%%%%%%%%%%%%%%%%%%%%%%%%%%%%%%%%%%%%%%%%%%%%%%%%%%%%%%%%%%%%%%
\subsection{$B^{(2)}_{g_2}=[F_1',T_{13}]$}
If $T_{13}=\emptyset$,
then this case is contained in the case~\ref{subsec:1}.
Thus we may assume $T_{13}\ne\emptyset$.
We have $T_{12}=\emptyset$.
\begin{lemma}\label{lem:2}
We have
$g_2=1$,
$T_{22}=\emptyset$ and
$A^{(2)}_1=\TP{[F_2',T_{23}]}$.
Moreover,
$[T_{24},S_3]^{\ast}=[\TW{k_{12}+k_{34}-2},k_{34}+1,\TP{T}_{14}]\TA\TW{\NS^{(2)}_1}$, $B^{(1)}_{g_1}=\TW{k_{12}}$.
\end{lemma}
\begin{proof}
Since $T_{12}=\emptyset$,
$[S_1,T_{11}]=[l_{22}+2,\TW{k_{12}-1}]$ by Lemma~\ref{lem:0}.
This means that $s_1=l_{22}+2$, $T_{11}=\TW{k_{12}-1}$.
By Proposition~\ref{prop:cres} and Lemma~\ref{lem:0},
$A^{(2)}_{g_2}=\TW{\NS^{(2)}_{g_2}}\TA[T_{14},k_{34}+1,\TW{k_{12}-1}]$.
Suppose $g_2=2$.
Since $A^{(2)}_2=F_2'$,
we get $\NS^{(2)}_2=1$, $T_{14}=\emptyset$, $k_{12}=1$ and $f_2=k_{34}+2$.
We have $g_1=1$ and $A^{(1)}_1=[T_{24},S_3]$.
By Proposition~\ref{prop:cres},
$[T_{24},S_3]^{\ast}=[l_{22}+2,\NS^{(1)}_1+1]$.
By Lemma~\ref{lem:0}, $T_{22}$ consists of ($-2$)-curves.
By Proposition~\ref{prop:cres},
we have $A^{(2)}_1=\TP{T}_{23}$, $B^{(2)}_1=\TP{T}_{22}$
and $T_{23}=[l_{22}+2,\TW{\NS^{(2)}_1-1}]$.
By Lemma~\ref{lem:0},
$[T_{24},S_3]^{\ast}=[\TW{k_{34}-1},k_{34}+2,l_{22}+2,\TW{\NS^{(2)}_1-1}]$.
Thus $[l_{22}+2,\NS^{(1)}_1+1]=[\TW{k_{34}-1},k_{34}+2,l_{22}+2,\TW{\NS^{(2)}_1-1}]$.
This shows $k_{34}=1$.
We have $[F_1',T_{13}]=[2]$ by Lemma~\ref{lem:0},
which contradicts $T_{13}\ne\emptyset$.
Hence $g_2=1$.

Suppose $T_{22}\ne\emptyset$.
We have $T_{23}=\emptyset$ and
$A^{(2)}_1=[T_{22},F_2']$.
Since $A^{(2)}_{1}=\TW{\NS^{(2)}_{1}}\TA[T_{14},k_{34}+1,\TW{k_{12}-1}]$
and $T_{22}=\TW{l_{22}}$,
we infer $T_{14}=\emptyset$,
$k_{12}=1$ and $f_2=k_{34}+2$.
We have $g_1=1$ and $A^{(1)}_1=[T_{24},S_3]$.
By Proposition~\ref{prop:cres}, $[T_{24},S_3]^{\ast}=[l_{22}+2,\NS^{(1)}_1+1]$.
By Lemma~\ref{lem:0}, $[\TW{k_{34}-1},k_{34}+2]=[l_{22}+2,\NS^{(1)}_1+1]$.
This shows $k_{34}=2$ and $l_{22}=0$, which is absurd.
Hence $T_{22}=\emptyset$.
We get $A^{(2)}_1=\TP{[F_2',T_{23}]}$.
By Lemma~\ref{lem:0},
we have
$[T_{24},S_3]^{\ast}=[\TW{k_{12}+k_{34}-2},k_{34}+1,\TP{T}_{14}]\TA\TW{\NS^{(2)}_1}$ and
$B^{(1)}_{g_1}=\TW{k_{12}}$.
\end{proof}

Case (1): $k_{34}=1$.
Suppose $T_{14}=\emptyset$.
We have $g_1=1$ and $A^{(1)}_1=[T_{24},S_3]$.
By Lemma~\ref{lem:0},
$[T_{24},S_3]=[\TW{\NS^{(1)}_1-1},k_{12}+2]$.
On the other hand,
$[T_{24},S_3]=[\NS^{(2)}_1+1,k_{12}+1]$ by Lemma~\ref{lem:2},
which is impossible.
Thus $T_{14}\ne\emptyset$.
By Lemma~\ref{lem:2},
$[T_{24},S_3]=[\NS^{(2)}_1+1,\TP{T}_{14}^{\ast}\TA\TW{1}^{\ast k_{12}}]$.
Thus $T_{24}\ne\emptyset$.
We have $g_1=2$ and $A^{(1)}_2=S_3$.
By Lemma~\ref{lem:0},
$\NS^{(1)}_2=1$ and $s_3=k_{12}+2$.
Either $A^{(1)}_1=T_{14}$ or $A^{(1)}_1=T_{24}$.
If $A^{(1)}_1=T_{14}$,
then $T_{14}^{\ast}=[\TP{T}_{24},\NS^{(1)}_1+1$] by Proposition~\ref{prop:cres}.
We get $[T_{24},S_3]=[\NS^{(2)}_1+1,\NS^{(1)}_1+1,T_{24}\TA\TW{1}^{\ast k_{12}}]$,
which is a contradiction.
Hence $A^{(1)}_1=T_{24}$ and $B^{(1)}_1=\TP{T}_{14}$.
By Proposition~\ref{prop:cres},
$T_{24}=\TW{\NS^{(1)}_1}\TA\TP{T}_{14}^{\ast}$.
Thus $[\TW{\NS^{(1)}_1}\TA\TP{T}_{14}^{\ast},S_3]=[\NS^{(2)}_1+1,\TP{T}_{14}^{\ast}\TA\TW{1}^{\ast k_{12}}]$.
Hence $\NS^{(1)}_1=1$.
We have $\TW{s_3-1}\TA[\TP{T}_{14},2]=[\TW{k_{12}},\TP{T}_{14}]\TA\TW{\NS^{(2)}_1}$.
This shows $s_3=k_{12}+\NS^{(2)}_1$ and $\NS^{(2)}_1=2$.
Thus $[\TP{T}_{14}^{\ast},2]=[2,\TP{T}_{14}^{\ast}]$.
There exists a positive integer $l$ such that $T_{14}^{\ast}=\TW{l}$
by Lemma~\ref{lem:adj} (iii).
Write $k=k_{12}$.
We have
$B^{(1)}_2=\TW{k}$,
$A^{(1)}_2=[k+2]$ and
$B^{(1)}_1=[l+1]$.
By Lemma~\ref{lem:0},
$B^{(2)}_1=[k+2,\TW{l-1}]$.
By Proposition~\ref{prop:cres},
we have
$A^{(1)}_1=[3,\TW{l-1}]$ and
$A^{(2)}_1=[2,l+2,\TW{k}]$.
The numerical data of
$C$ is equal to
$\{(((l+1)(k+1),l(k+1),(k+1)_l),\;\;((l(k+1)+1)_2,(k+1)_l)\}$,
which coincides with the data 2 with $a=l$, $b=k+1$.

Case (2): $k_{34}>1$, $T_{14}=\emptyset$.
We have $g_1=1$ and $A^{(1)}_1=[T_{24},S_3]$.
By Proposition~\ref{prop:cres} and Lemma~\ref{lem:2},
$[T_{24},S_3]^{\ast}=[B^{(1)}_1,\NS^{(1)}_1+1]=[\TW{k_{12}},\NS^{(1)}_1+1]$.
On the other hand,
$[T_{24},S_3]^{\ast}=[\TW{k_{12}+k_{34}-2},k_{34}+2,\TW{\NS^{(2)}_1-1}]$
by Lemma~\ref{lem:2}.
Hence $k_{34}=2$, $\NS^{(2)}_1=1$ and $\NS^{(1)}_1=3$.
Write $k=k_{12}$.
We have
$B^{(1)}_1=\TW{k}$.
By Lemma~\ref{lem:0},
$B^{(2)}_1=[k+1,2]$.
By Proposition~\ref{prop:cres},
we have
$A^{(1)}_1=[2,2,k+2]$ and
$A^{(2)}_1=[4,\TW{k-1}]$.
The numerical data of
$C$ is equal to
$\{((k+1)_3),\;\;(2k+1,k_2)\}$,
which coincides with the data 1 with $a=1$, $b=k+1$.

Case (3): $k_{34}>1$, $T_{14}\ne\emptyset$.
By Lemma~\ref{lem:2},
$[T_{24},S_3]=[\NS^{(2)}_1+1,\TP{T}_{14}^{\ast}\TA\TW{k_{34}-1},k_{12}+k_{34}]$.
We have $T_{24}=[\NS^{(2)}_1+1,\TP{T}_{14}^{\ast}\TA\TW{k_{34}-1}]$
and $s_3=k_{12}+k_{34}$.
We infer $T_{24}\ne\emptyset$ and $g_1=2$.
Since $A^{(1)}_2=S_3$, we get $\NS^{(1)}_2=1$ and $k_{34}=2$
by Lemma~\ref{lem:0}.
Either $B^{(1)}_1=\TP{T}_{14}$ or $B^{(1)}_1=\TP{T}_{24}$.
If $B^{(1)}_1=\TP{T}_{24}$, then
$\TP{T}_{14}^{\ast}=[\NS^{(1)}_1+1,T_{24}]$ by Proposition~\ref{prop:cres}.
Thus $T_{24}=[\NS^{(2)}_1+1,\NS^{(1)}_1+1,T_{24}\TA\TW{1}]$, which is absurd.
Hence $B^{(1)}_1=\TP{T}_{14}$.
By Proposition~\ref{prop:cres},
$T_{24}=A^{(1)}_1=\TW{\NS^{(1)}_1}\TA\TP{T}_{14}^{\ast}$.
This means that $\NS^{(1)}_1=2$, $\NS^{(2)}_1=1$ and
$\TW{1}\TA\TP{T}_{14}^{\ast}=\TP{T}_{14}^{\ast}\TA\TW{1}$.
Thus $[T_{14},2]=[2,T_{14}]$.
There exists a positive integer $l$ such that $T_{14}=\TW{l}$.
Write $k=k_{12}$.
We have
$B^{(1)}_2=\TW{k}$,
$A^{(1)}_2=[k+2]$,
$B^{(1)}_1=\TW{l}$.
By Lemma~\ref{lem:0},
$B^{(2)}_1=[k+1,l+2]$.
By Proposition~\ref{prop:cres},
we have
$A^{(1)}_1=[2,l+2]$ and
$A^{(2)}_1=[3,\TW{l-1},3,\TW{k-1}]$.
The numerical data of
$C$ is equal to
$\{(((l+1)(k+1))_2,(k+1)_{l+1}),\;\;((l+1)(k+1)+k,l(k+1)+k,(k+1)_l,k)\}$,
which coincides with the data 1 with $a=l+1$, $b=k+1$.
%
%
%
%%%%%%%%%%%%%%%%%%%%%%%%%%%%%%%%%%%%%%%%%%%%%%%%%%%%%%%%%%%%%%%%%%%%%%%%%%%%%%%
\subsection{$B^{(2)}_{g_2}=\protect\TP{[T_{22},F_2']}$}
We have $T_{23}=\emptyset$.
We may assume $T_{11}\ne\emptyset$
because
this case is contained in the case~\ref{subsec:1}
if $T_{11}=\emptyset$.
\begin{lemma}\label{lem:3}
We have $g_1=1$, $T_{24}=\emptyset$, $A^{(1)}_1=[T_{14},S_3]$,
$f_2=2$, $s_3=k_{34}+1$, $f_1=l_{22}+3$,
and
$[F_1',T_{13}]^{\ast}=\TW{\NS^{(1)}_1}\TA[T_{12},k_{12}+1,\TW{l_{22}+k_{12}-1}]$.
\end{lemma}
\begin{proof}
By Lemma~\ref{lem:0},
we have $[T_{24},S_3]=[\TW{f_2-2},k_{34}+1]$.
This shows that $T_{24}=\TW{f_2-2}$ and $s_3=k_{34}+1$.
Suppose $g_1=2$.
Since $A^{(1)}_2=S_3$,
we have $[k_{34}+1]=\TW{\NS^{(1)}_2}\TA[T_{12},k_{12}+1,\TW{l_{22}}]$
by Lemma~\ref{lem:0}.
This means that $\NS^{(1)}_2=1$, $T_{12}=\emptyset$, $l_{22}=0$ and $k_{34}=k_{12}+1$.
We infer $g_2=1$ and $A^{(2)}_1=\TP{[F_1',T_{13}]}$.
By Proposition~\ref{prop:cres},
$[F_1',T_{13}]^{\ast}=\TP{A}^{(2)\ast}_1=[\NS^{(2)}_1+1,F_2']$.
By Lemma~\ref{lem:0},
$[\NS^{(2)}_1+1,F_2']=[T_{14},k_{12}+2,\TW{k_{12}-1}]$.
Because $T_{14}\ne\emptyset$, we have $k_{12}=1$.
By Lemma~\ref{lem:0},
we obtain $[S_1,T_{11}]=[2]$,
which contradicts $T_{11}\ne\emptyset$.
Hence $g_1=1$.

Suppose $T_{24}\ne\emptyset$.
We have $T_{14}=\emptyset$ and $A^{(1)}_1=[T_{24},S_3]=[\TW{f_2-2},k_{34}+1]$.
Thus $f_2>2$.
By Lemma~\ref{lem:0},
$[F_1',T_{13}]=[k_{12}+1,\TW{k_{34}-1}]$.
This shows $f_1=k_{12}+1$.
By Proposition~\ref{prop:cres},
$A^{(2)}_{g_2}=[\TW{\NS^{(2)}_{g_2}-1},l_{22}+3,\TW{f_2-2}]$.
Since $A^{(2)}_{g_2}=\TP{[F_1',\ldots]}$, we have $f_1=2$ and $k_{12}=1$.
By Lemma~\ref{lem:0},
$A^{(1)}_{1}=\TW{\NS^{(1)}_{1}}\TA[T_{12},\TW{l_{22}+1}]$.
Thus $[\TW{f_2-2},k_{34}+1]=\TW{\NS^{(1)}_{1}}\TA[T_{12},\TW{l_{22}+1}]$.
If $T_{12}\ne\emptyset$ or $l_{22}>0$, then $k_{34}=1$ and
$\TW{f_2-2}=\TW{\NS^{(1)}_{1}}\TA[T_{12},\TW{l_{22}}]$, which is impossible.
Thus $T_{12}=\emptyset$ and $l_{22}=0$.
By Lemma~\ref{lem:0},
$[S_1,T_{11}]=[2]$, which contradicts $T_{11}\ne\emptyset$.
Hence $T_{24}=\emptyset$ and $A^{(1)}_1=[T_{14},S_3]$.
We have $f_2=2$.
By Proposition~\ref{prop:cres},
$A^{(2)}_{g_2}=[\TW{\NS^{(2)}_{g_2}-1},l_{22}+3]$.
This shows $f_1=l_{22}+3$.
By Lemma~\ref{lem:0},
$[F_1,T_{13}]^{\ast}=\TW{\NS^{(1)}_1}\TA[T_{12},k_{12}+1,\TW{l_{22}+k_{12}-1}]$.\end{proof}

Case (1): $g_2=1$.
If $T_{13}=\emptyset$, then
$\TW{f_1-1}=\TW{\NS^{(1)}_1}\TA[T_{12},k_{12}+1,\TW{l_{22}+k_{12}-1}]$
by Lemma~\ref{lem:3},
which is absurd.
Thus $T_{13}\ne\emptyset$.
We have $T_{12}=\emptyset$ and $A^{(2)}_1=\TP{[F_1',T_{13}]}$.
By Lemma~\ref{lem:3},
$\TP{A}^{(2)\ast}_1=\TW{\NS^{(1)}_1}\TA[k_{12}+1,\TW{l_{22}+k_{12}-1}]$.
By Proposition~\ref{prop:cres},
$\TP{A}^{(2)\ast}_1=[\NS^{(2)}_1+1,\TW{l_{22}+1}]$.
It follows that $\NS^{(1)}_1+k_{12}=3$.
By Lemma~\ref{lem:0},
$[S_1,T_{11}]=[l_{22}+2,\TW{k_{12}-1}]$.
Since $T_{11}\ne\emptyset$,
we see $k_{12}=2$, $\NS^{(1)}_1=1$ and $\NS^{(2)}_1=3$.
Put $k=l_{22}+1$.
We have
$k\ge 1$,
$B^{(1)}_1=[k+1,2]$ and
$B^{(2)}_1=\TW{k}$.
By Proposition~\ref{prop:cres},
we get
$A^{(1)}_1=[4,\TW{k-1}]$ and
$A^{(2)}_1=[2,2,k+2]$.
The numerical data of
$C$ is equal to
$\{(2k+1,k_2),\;\;((k+1)_3)\}$,
which coincides with the data 1 with $a=1$, $b=k+1$.

Case (2): $g_2=2$.
We have $A^{(2)}_2=F_1'$, $T_{12}\ne\emptyset$ and $T_{13}\ne\emptyset$.
\begin{lemma}\label{lem:3-2}
We have
$B^{(2)}_1=\TP{T}_{12}$, $A^{(2)}_1=\TP{T}_{13}$,
$[\NS^{(2)}_1+1,T_{12}\ast\TW{l_{22}+2}]=[\TW{\NS^{(1)}_1}\TA T_{12},k_{12}+1,\TW{l_{22}+k_{12}-1}]$ and
$3=\NS^{(1)}_1+k_{12}$.
\end{lemma}
\begin{proof}
Either $B^{(2)}_1=\TP{T}_{12}$ or $B^{(2)}_1=T_{13}$.
Suppose $B^{(2)}_1=T_{13}$.
By Proposition~\ref{prop:cres},
$T_{12}=A^{(2)}_1=\TW{\NS^{(2)}_1}\TA T_{13}^{\ast}$.
By Lemma~\ref{lem:3},
$[l_{22}+3,T_{13}]=[\TW{\NS^{(1)}_1}\TA\TW{\NS^{(2)}_1}\TA T_{13}^{\ast},k_{12}+1,\TW{l_{22}+k_{12}-1}]^{\ast}=\TW{1}^{\ast l_{22}+k_{12}-1}\TA\TW{k_{12}}\TA[T_{13},\NS^{(2)}_1+1,\NS^{(1)}_1+1]$, which is impossible.
Thus $B^{(2)}_1=\TP{T}_{12}$ and $A^{(2)}_1=\TP{T}_{13}$.
By Proposition~\ref{prop:cres},
$T_{13}^{\ast}=\TP{A}^{(2)\ast}_{1}=[\NS^{(2)}_1+1,T_{12}]$.
By Lemma~\ref{lem:3},
$[\NS^{(2)}_1+1,T_{12}\ast\TW{l_{22}+2}]=[\TW{\NS^{(1)}_1}\TA T_{12},k_{12}+1,\TW{l_{22}+k_{12}-1}]$.
Hence $3=\NS^{(1)}_1+k_{12}$.
\end{proof}

Case (2--1): $k_{12}=1$.
By Lemma~\ref{lem:3-2},
we have $\NS^{(1)}_1=2$, $\NS^{(2)}_1=1$ and
$T_{12}\TA\TW{1}=\TW{1}\TA T_{12}$.
Thus $[2,T_{12}^{\ast}]=[T_{12}^{\ast},2]$.
There exists a positive integer $l'$ such that $T_{12}^{\ast}=\TW{l'}$.
Hence $T_{12}=[l'+1]$.
Put $l=l'-1$ and $k=l_{22}+1$.
By Lemma~\ref{lem:0},
$B^{(1)}_1=[k+2,\TW{l}]$.
Since $T_{11}\ne\emptyset$,
we see $l\ge 1$.
We have
$B^{(2)}_2=\TW{k}$,
$A^{(2)}_2=[k+2]$ and
$B^{(2)}_1=[l+2]$.
By Proposition~\ref{prop:cres},
we see
$A^{(1)}_1=[2,l+3,\TW{k}]$ and
$A^{(2)}_1=[3,\TW{l}]$.
The numerical data of
$C$ is equal to
$\{(((l+1)(k+1)+1)_2,(k+1)_{l+1}),\;\;((l+2)(k+1),(l+1)(k+1),(k+1)_{l+1})\}$,
which coincides with the data 2 with $a=l+1$, $b=k+1$.

Case (2--2): $k_{12}=2$.
By Lemma~\ref{lem:3-2},
$\NS^{(1)}_1=1$ and $[\NS^{(2)}_1,T_{12}]=[T_{12},2]$.
We have $\NS^{(2)}_1\ge2$.
Since $T_{12}^{\ast}\TA\TW{\NS^{(2)}_1-1}=\TW{1}\TA T_{12}^{\ast}$,
we see $\NS^{(2)}_1=2$.
There exists a positive integer $l$ such that $T_{12}=\TW{l}$.
Put $k=l_{22}+1$.
We have
$B^{(2)}_2=\TW{k}$,
$A^{(2)}_2=[k+2]$ and
$B^{(2)}_1=\TW{l}$.
By Lemma~\ref{lem:0},
$B^{(1)}_1=[k+1,l+2]$.
By Proposition~\ref{prop:cres},
we have
$A^{(1)}_1=[3,\TW{l-1},3,\TW{k-1}]$ and
$A^{(2)}_1=[2,l+2]$.
The numerical data of
$C$ is equal to
$\{((l+1)(k+1)+k,l(k+1)+k,(k+1)_{l},k),\;\;(((l+1)(k+1))_2,(k+1)_{l+1})\}$,
which coincides with the data 1 with $a=l+1$, $b=k+1$.
%%%%%%%%%%%%%%%%%%%%%%%%%%%%%%%%%%%%%%%%%%%%%%%%%%%%%%%%%%%%%%%%%%%%%%%%%%%%%%%
\subsection{$B^{(2)}_{g_2}=[F_2',T_{23}]$}

We have $T_{22}=\emptyset$.
We may assume $T_{11}\ne\emptyset\ne T_{23}$;
otherwise this case is contained in another case.

Case (1): $g_2=1$.
We show the following lemma.
\begin{lemma}\label{lem:4-1}
We have $T_{12}=\emptyset$, $A_1^{(2)}=\TP{[F_1',T_{13}]}$,
$f_2=2$, $B^{(1)}_{g_1}=\TW{k_{12}}$, $k_{12}\ge 2$
and
$[T_{14},k_{34}+1,\TW{k_{12}-2}]=[\NS^{(2)}_1+1,\TP{T_{23}}]$.
\end{lemma}
\begin{proof}
Suppose $T_{12}\ne\emptyset$.
We have $A^{(2)}_1=[T_{12},F_1']$ and $T_{13}=\emptyset$.
By Lemma~\ref{lem:0},
$[T_{14},k_{34}+1,\TW{k_{12}-1}]=\TW{f_1-1}$.
This shows that $k_{34}=1$ and $T_{14}=\TW{f_1-k_{12}-1}$.
By Lemma~\ref{lem:0},
$A^{(1)}_{g_1}=\TW{\NS^{(1)}_{g_1}}\TA [T_{12},k_{12}+1]$.
Thus $A^{(1)}_{g_1}$ contains at least two irreducible components.
It follows that $g_1=1$.
Either $A^{(1)}_1=[T_{14},S_3]$ or $A^{(1)}_1=[T_{24},S_3]$.
Because $T_{14}$ consists of ($-2$)-curves,
the latter case must occur.
We infer $T_{24}=\TW{\NS^{(1)}_1}\TA T_{12}$.
By Proposition~\ref{prop:cres} and Lemma~\ref{lem:0},
$[T_{12},F_1']=A^{(2)}_1=\TW{\NS^{(2)}_1}\TA[F_2',T_{23}]^{\ast}=\TW{\NS^{(2)}_1}\TA[T_{24},S_3]$.
We have $T_{12}=\TW{\NS^{(2)}_1}\TA T_{24}$.
Thus
$T_{24}=\TW{\NS^{(1)}_1}\TA\TW{\NS^{(2)}_1}\TA T_{24}$, which is impossible.
Hence $T_{12}=\emptyset$.

We have $A_1^{(2)}=\TP{[F_1',T_{13}]}$.
By Lemma~\ref{lem:0},
$[T_{14},k_{34}+1,\TW{k_{12}-1}]=\TP{A}^{(2)\ast}_1$ and
$B^{(1)}_{g_1}=[S_1,T_{11}]=\TW{k_{12}}$.
We infer $k_{12}\ge 2$.
By Proposition~\ref{prop:cres},
$[T_{14},k_{34}+1,\TW{k_{12}-1}]=[\NS^{(2)}_1+1,\TP{T_{23}},F_2']$.
This shows that
$f_2=2$ and
$[T_{14},k_{34}+1,\TW{k_{12}-2}]=[\NS^{(2)}_1+1,\TP{T_{23}}]$.
\end{proof}

Case (1--1): $g_1=1$.
Suppose $T_{24}\ne\emptyset$.
We have $T_{14}=\emptyset$ and $A^{(1)}_1=[T_{24},S_3]$.
By Lemma~\ref{lem:0} and Lemma~\ref{lem:4-1},
$[\TW{k_{34}},T_{23}]=A^{(1)\ast}_1$.
By Proposition~\ref{prop:cres},
we have $[\TW{k_{34}},T_{23}]=[\TW{k_{12}},\NS^{(1)}_1+1]$.
By Lemma~\ref{lem:4-1},
$[T_{23},\NS^{(2)}_1+1]=[\TW{k_{12}-2},k_{34}+1]$.
We infer
$[\TW{k_{12}},\NS^{(1)}_1+1,\NS^{(2)}_1+1]=[\TW{k_{12}+k_{34}-2},k_{34}+1]$.
This means that $k_{34}=3$ and $\NS^{(1)}_1=1$.
By Proposition~\ref{prop:cres},
$[T_{24},S_3]=[k_{12}+2]$, which is absurd.
Thus $T_{24}=\emptyset$.

We have $A^{(1)}_1=[T_{14},S_3]$.
By Lemma~\ref{lem:0},
we get $[T_{14},S_3]=[\TW{\NS^{(1)}_1-1},k_{12}+2]$.
Hence $s_3=k_{12}+2$ and $T_{14}=\TW{\NS^{(1)}_1-1}$.
By Proposition~\ref{prop:cres} and Lemma~\ref{lem:4-1},
$[F_2',T_{23},\NS^{(2)}_1+1]=A^{(2)\ast}_1=\TP{[F_1',T_{13}]}^{\ast}$.
By Lemma~\ref{lem:0},
$[F_2',T_{23},\NS^{(2)}_1+1]=[\TW{k_{12}-1},k_{34}+1,\TW{\NS^{(1)}_1-1}]$.
Since $k_{12}\ge 2$,
we see
$[T_{23},\NS^{(2)}_1+1]=[\TW{k_{12}-2},k_{34}+1,\TW{\NS^{(1)}_1-1}]$.
By Lemma~\ref{lem:0},
$[\TW{k_{34}-1},F_2',T_{23}]=\TW{s_3-1}=\TW{k_{12}+1}$.
Thus $T_{23}=\TW{k_{12}-k_{34}+1}$.
Hence
$[\TW{k_{12}-k_{34}+1},\NS^{(2)}_1+1]=[\TW{k_{12}-2},k_{34}+1,\TW{\NS^{(1)}_1-1}]$.
We infer $4=k_{34}+\NS^{(1)}_1$.

Case ($\text{1--1}_\text{a}$): $k_{34}=1$.
We have $\NS^{(1)}_1=3$, $\NS^{(2)}_1=1$ and $[F_2',T_{23}]=\TW{k_{12}+1}$.
Put $k=k_{12}-1$.
By Lemma~\ref{lem:4-1},
$k\ge 1$ and $B^{(1)}_1=\TW{k+1}$.
We have
$B^{(2)}_1=\TW{k+2}$.
By Proposition~\ref{prop:cres},
$A^{(1)}_1=[2,2,k+3]$ and
$A^{(2)}_1=[k+4]$.
The numerical data of $C$ is equal to $\{(k+3),\, ((k+2)_3)\}$,
which coincides with the data 3 with $a=1$, $b=k+2$.

Case ($\text{1--1}_\text{b}$): $k_{34}>1$.
If $\NS^{(1)}_1=2$, then 
$[\TW{k_{12}-k_{34}+1},\NS^{(2)}_1+1]=[\TW{k_{12}-2},3,2]$,
which is impossible.
We have $\NS^{(1)}_1=1$, $k_{34}=3$ and  $\NS^{(2)}_1=3$.
Put $k=k_{12}-2$.
Since $[F_2',T_{23}]=\TW{k+1}$, we see $k\ge 1$.
We have
$B^{(1)}_1=\TW{k+2}$ and
$B^{(2)}_1=\TW{k+1}$.
By Proposition~\ref{prop:cres},
we obtain
$A^{(1)}_1=[k+4]$ and
$A^{(2)}_1=[2,2,k+3]$.
The numerical data of $C$ is equal to $\{((k+2)_3),\, (k+3)\}$,
which coincides with the data 3 with $a=1$, $b=k+2$.

Case (1--2): $g_1=2$.
We have $A^{(1)}_2=S_3$, $T_{14}\ne\emptyset$ and $T_{24}\ne\emptyset$.
\begin{lemma}\label{lem:4-1-2}
We have $\NS^{(1)}_2=1$, $s_3=k_{12}+2$,
$A^{(1)}_1=T_{14}$,
$B^{(1)}_1=\TP{T}_{24}$,
$[\TW{k_{12}+k_{34}-2},k_{34}+1,T_{24}^{\ast}\TA\TW{\NS^{(1)}_1}]=[\TW{k_{12}+1}\TA T_{24}^{\ast},\NS^{(2)}_1+1]$,
$k_{34}+\NS^{(1)}_1=3$.
\end{lemma}
\begin{proof}
By Lemma~\ref{lem:0},
we get
$\NS^{(1)}_2=1$,
$s_3=k_{12}+2$
and
$[\TW{k_{34}},T_{23}]=\TW{k_{12}+1}\TA T_{24}^{\ast}$.
Either $B^{(1)}_1=\TP{T}_{24}$ or $B^{(1)}_1=\TP{T}_{14}$.
Suppose $B^{(1)}_1=\TP{T}_{14}$.
We have $A^{(1)}_1=T_{24}$.
By Proposition~\ref{prop:cres},
$[\TW{k_{34}},T_{23}]=[\TW{k_{12}+1}\TA \TP{T}_{14},\NS^{(1)}_1+1]$.
By Lemma~\ref{lem:4-1},
we have
$[T_{14},k_{34}+1,\TW{k_{12}+k_{34}-2}]=[\NS^{(2)}_1+1,\TP{T}_{23},\TW{k_{34}}]=[\NS^{(2)}_1+1,\NS^{(1)}_1+1,T_{14}\TA\TW{k_{12}+1}]$.
Thus $k_{34}=3$.
It follows that
$[T_{14},4,2]=[\NS^{(2)}_1+1,\NS^{(1)}_1+1,T_{14}\TA\TW{1}]$,
which is impossible.
Hence $A^{(1)}_1=T_{14}$, $B^{(1)}_1=\TP{T}_{24}$.
By Proposition~\ref{prop:cres},
$T_{14}=\TW{\NS^{(1)}_1}\TA \TP{T}_{24}^{\ast}$.
By Lemma~\ref{lem:4-1},
$[T_{23},\NS^{(2)}_1+1]=[\TW{k_{12}-2},k_{34}+1,T_{24}^{\ast}\TA\TW{\NS^{(1)}_1}]$.
Thus
$[\TW{k_{12}+1}\TA T_{24}^{\ast},\NS^{(2)}_1+1]=[\TW{k_{12}+k_{34}-2},k_{34}+1,T_{24}^{\ast}\TA\TW{\NS^{(1)}_1}]$.
Hence $3=k_{34}+\NS^{(1)}_1$.
\end{proof}

Case (1--$2_\text{a}$): $k_{34}=1$.
By Lemma~\ref{lem:4-1-2},
we have $\NS^{(1)}_1=2$, $\NS^{(2)}_1=1$ and
$[\TW{k_{12}+1}\TA T_{24}^{\ast},2]=[\TW{k_{12}},T_{24}^{\ast}\TA\TW{2}]$.
Thus 
$\TW{1}\TA T_{24}^{\ast}=T_{24}^{\ast}\TA\TW{1}$.
Hence $[T_{24},2]=[2,T_{24}]$.
There exists a positive integer $l$ such that $T_{24}=\TW{l}$.
Put $k=k_{12}-1$.
By Lemma~\ref{lem:4-1},
we have $k\ge1$, $[S_1,T_{11}]=B^{(1)}_{2}=\TW{k+1}$.
By Lemma~\ref{lem:4-1-2}, $A^{(1)}_2=[k+3]$.
Since $B^{(1)}_1=\TW{l}$,
we get $A^{(1)}_1=[2,l+2]$ by Proposition~\ref{prop:cres}.
By Lemma~\ref{lem:0},
we infer $B^{(2)}_1=[\TW{k+1},l+2]$.
By Proposition~\ref{prop:cres},
$A^{(2)}_1=[3,\TW{l-1},k+3]$.
It follows that
the numerical data of $C$ is equal to
$\{(((l+1)(k+2))_2,(k+2)_{l+1}),\;\;((l+1)(k+2)+1,l(k+2)+1,(k+2)_l)\}$,
which coincides with the data 3 with $a=l+1$, $b=k+2$.

Case (1--$2_\text{b}$): $k_{34}=2$.
By Lemma~\ref{lem:4-1-2},
we have $\NS^{(1)}_1=1$ and
$[\TW{k_{12}},3,T_{24}^{\ast}\TA\TW{1}]=[\TW{k_{12}+1}\TA T_{24}^{\ast},\NS^{(2)}_1+1]$.
Thus $[2,T_{24}^{\ast}]=[T_{24}^{\ast},\NS^{(2)}_1]$.
Hence $T_{24}\TA\TW{1}=\TW{\NS^{(2)}_1-1}\TA T_{24}$.
This shows that $\NS^{(2)}_1=2$ and $[2,T_{24}^{\ast}]=[T_{24}^{\ast},2]$.
There exists a positive integer $l$ such that $T_{24}^{\ast}=\TW{l}$.
We have $T_{24}=[l+1]$.
Put $k=k_{12}-1$.
By Lemma~\ref{lem:4-1},
we see $k\ge1$, $[S_1,T_{11}]=B^{(1)}_2=\TW{k+1}$.
We have $A^{(1)}_2=[k+3]$ by Lemma~\ref{lem:4-1-2}.
Since $B^{(1)}_1=[l+1]$,
we get $A^{(1)}_1=[3,\TW{l-1}]$ by Proposition~\ref{prop:cres}.
By Lemma~\ref{lem:0},
we infer $B^{(2)}_1=[\TW{k},3,\TW{l-1}]$.
By Proposition~\ref{prop:cres},
$A^{(2)}_1=[2,l+2,k+2]$.
It follows that
the numerical data of $C$ is equal to
$\{((l+1)(k+2),l(k+2),(k+2)_{l}),\;\;((l(k+2)+k+1)_2,(k+2)_l,k+1)\}$,
which coincides with the data 4 with $a=l$, $b=k+2$.

Case (2): $g_2=2$.
We have $T_{12}\ne\emptyset\ne T_{13}$ and $A^{(2)}_2=F_1'$.
\begin{lemma}\label{lem:4-2}
The following assertions hold.
\begin{enumerate}
\item[(i)]
$g_1=1$,
$T_{24}=\emptyset$,
$A^{(1)}_1=[T_{14},S_3]$,
$B^{(2)}_1=\TP{T}_{12}$,
$A^{(2)}_1=\TP{T}_{13}$.
\item[(ii)]
$k_{12}\ge2$,
$f_1\ge4$, $f_2=2$, $s_3=k_{12}+1=k_{34}+f_1-2$,
$k_{34}+\NS^{(1)}_1=3$,
$\NS^{(2)}_1=k_{34}$,
$\NS^{(2)}_2=1$.
\item[(iii)]
$T_{13}=T_{12}^{\ast}\TA\TW{\NS^{(2)}_1}$,
$T_{14}=\TW{\NS^{(1)}_1}\TA T_{12}$,
$T_{23}=\TW{f_1-3}$,
$[F_1',T_{12}^{\ast}\TA\TW{\NS^{(2)}_1}]=[\TW{1}^{\ast k_{12}-1}\TA\TW{k_{34}}\TA T_{12}^{\ast},\NS^{(1)}_1+1]$.
\end{enumerate}
\end{lemma}
\begin{proof}
By Proposition~\ref{prop:cres},
$[F_2',T_{23},\NS^{(2)}_2+1]=\TW{f_1-1}$.
This shows that $f_1\ge4$,
$f_2=2$, $\NS^{(2)}_2=1$ and $T_{23}=\TW{f_1-3}$.
By Lemma~\ref{lem:0},
$[T_{24},S_3]=[k_{34}+f_1-2]$.
We have $T_{24}=\emptyset$ and $s_3=k_{34}+f_1-2$.
It follows that $g_1=1$ and $A^{(1)}_1=[T_{14},S_3]$.
By Lemma~\ref{lem:0},
$[T_{14},S_3]=\TW{\NS^{(1)}_1}\TA[T_{12},k_{12}+1]$.
Since $T_{12}\ne\emptyset$,
we obtain $\emptyset\ne\TW{\NS^{(1)}_1}\TA T_{12}=T_{14}$ and $k_{12}+1=s_3=k_{34}+f_1-2$.
We have $k_{12}\ge2$ and  $T_{14}^{\ast}=[T_{12}^{\ast},\NS^{(1)}_1+1]$.
By Lemma~\ref{lem:0},
$[F_1',T_{13}]=[\TW{1}^{\ast k_{12}-1}\TA\TW{k_{34}}\TA T_{12}^{\ast},\NS^{(1)}_1+1]$.
Either $B^{(2)}_1=T_{13}$ or $B^{(2)}_1=\TP{T}_{12}$.
If $B^{(2)}_1=T_{13}$, then  $A^{(2)}_1=T_{12}$.
By Proposition~\ref{prop:cres},
$[T_{13},\NS^{(2)}_1+1]=T_{12}^{\ast}$.
Hence
$[F_1',T_{13}]=[\TW{1}^{\ast k_{12}-1}\TA\TW{k_{34}}\TA T_{13},\NS^{(2)}_1+1,\NS^{(1)}_1+1]$,
which is impossible.
Thus $B^{(2)}_1=\TP{T}_{12}$, $A^{(2)}_1=\TP{T}_{13}$.
By Proposition~\ref{prop:cres},
$T_{13}=T_{12}^{\ast}\TA\TW{\NS^{(2)}_1}$.
We have
$[F_1',T_{12}^{\ast}\TA\TW{\NS^{(2)}_1}]=[\TW{1}^{\ast k_{12}-1}\TA\TW{k_{34}}\TA T_{12}^{\ast},\NS^{(1)}_1+1]$.
This means that
$\NS^{(2)}_1=k_{34}$ and
$[\NS^{(2)}_1+1,T_{12}\TA\TW{f_1-1}]=[\TW{\NS^{(1)}_1}\TA T_{12},k_{34}+1,\TW{k_{12}-1}]$.
We have $f_1=\NS^{(1)}_1+k_{12}$.
Hence $k_{34}+\NS^{(1)}_1=3$.
\end{proof}

Case (2--1): $k_{34}=1$.
By Lemma~\ref{lem:4-2},
$k_{12}=f_1-2$,
$\NS^{(1)}_1=2$,
$\NS^{(2)}_1=1$,
$[F_1',T_{12}^{\ast}\TA\TW{1}]=[\TW{1}^{\ast k_{12}}\TA T_{12}^{\ast},3]$.
We have
$[2,T_{12}\TA\TW{f_1-1}]=[\TW{2}\TA T_{12},\TW{f_1-2}]$.
This shows $T_{12}\TA\TW{1}=\TW{1}\TA T_{12}$.
There exists a positive integer $l$ such that $T_{12}=[l+1]$.
Put $k=k_{12}-1$.
By Lemma~\ref{lem:4-2},
we have $k\ge1$, $B^{(2)}_1=\TP{T}_{12}=[l+1]$.
Furthermore,
we see
$A^{(2)}_1=\TP{T}_{13}=[3,\TW{l-1}]$,
$A^{(2)}_2=[k+3]$,
$A^{(1)}_1=[T_{14},S_3]=[2,l+2,k+2]$,
$B^{(2)}_2=[F_2',T_{23}]=\TW{k+1}$.
By Proposition~\ref{prop:cres},
we obtain
$[B^{(1)}_1,3]=A^{(1)\ast}_1=\TW{k+1}\TA\TW{l+1}\TA\TW{1}$.
Hence $B^{(1)}_1=[\TW{k},3,\TW{l-1}]$.
The numerical data of $C$ is equal to
$\{((l(k+2)+k+1)_2,(k+2)_l,k+1),\;\;((l+1)(k+2),l(k+2),(k+2)_l)\}$,
which coincides with the data 4 with $a=l$, $b=k+2$.

Case (2--2): $k_{34}>1$.
By Lemma~\ref{lem:4-2},
we get
$k_{34}=2$,
$\NS^{(1)}_1=1$,
$\NS^{(2)}_1=2$,
$s_3=k_{12}+1=f_1\ge4$,
$[F_1',T_{12}^{\ast}\TA\TW{2}]=[k_{12}+1,\TW{1}\TA T_{12}^{\ast},2]$.
We have $T_{12}^{\ast}\TA\TW{1}=\TW{1}\TA T_{12}^{\ast}$.
There exists a positive integer $l$ such that $T_{12}^{\ast}=[l+1]$.
Put $k=k_{12}-2$.
We have $k\ge 1$,
$B^{(2)}_1=\TP{T}_{12}=\TW{l}$.
Moreover, we see
$A^{(2)}_1=\TP{T}_{13}=[2,l+2]$,
$A^{(2)}_2=[k+3]$,
$A^{(1)}_1=[T_{14},S_3]=[3,\TW{l-1},k+3]$,
$B^{(2)}_2=[F_2',T_{23}]=\TW{k+1}$.
By Proposition~\ref{prop:cres},
we obtain
$[B^{(1)}_1,2]=A^{(1)\ast}_1=\TW{k+2}\TA[l+1,2]$.
Hence $B^{(1)}_1=[\TW{k+1},l+2]$.
The numerical data of $C$ is equal to
$\{((l+1)(k+2)+1,l(k+2)+1,(k+2)_{l}),\;\;(((l+1)(k+2))_2,(k+2)_{l+1})\}$,
which coincides with the data 3 with $a=l+1$, $b=k+2$.
\begin{figure}
\begin{center}
\includegraphics{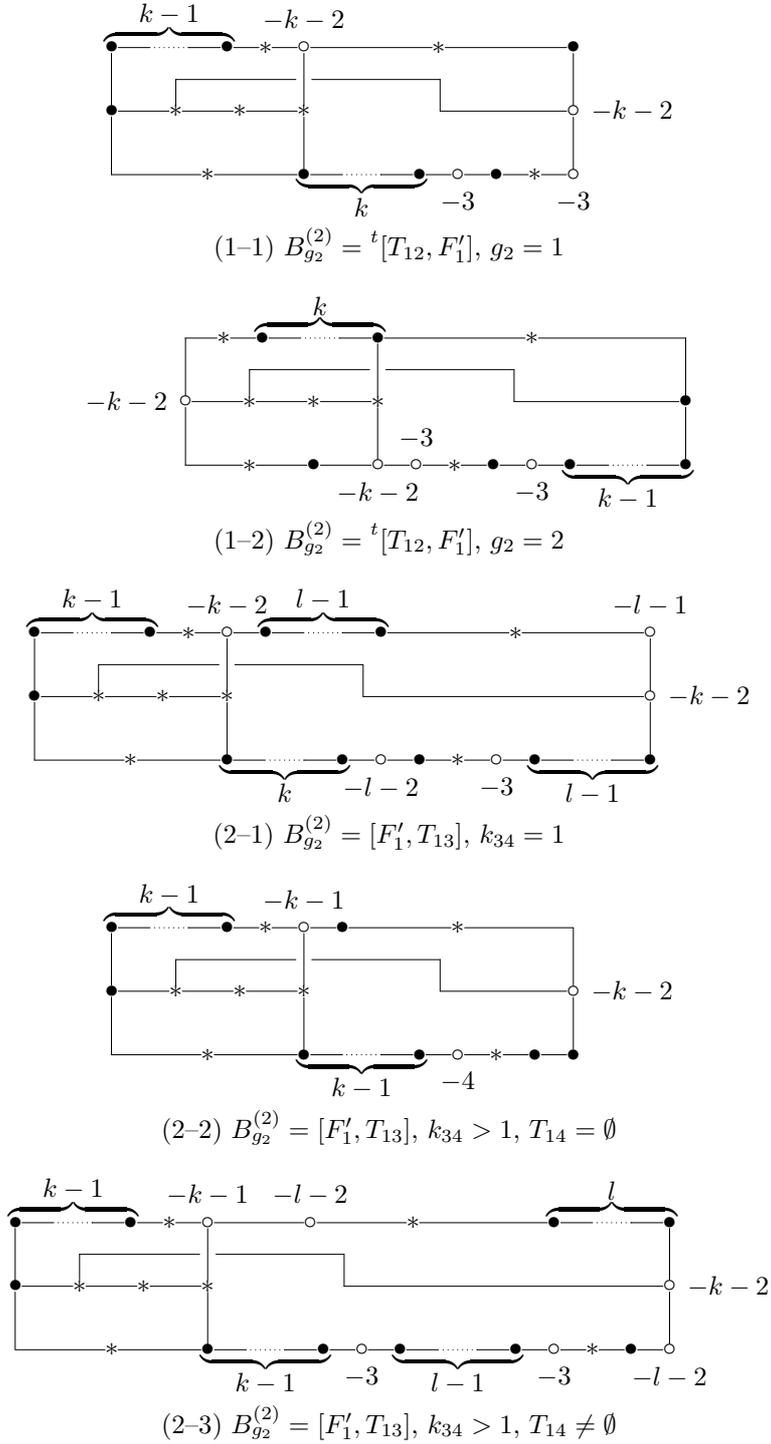}
\end{center}
\caption{The dual graph of $D+E_1+E_2$}\label{figf}
\end{figure}
\setcounter{figure}{3}
\begin{figure}
\begin{center}
\includegraphics{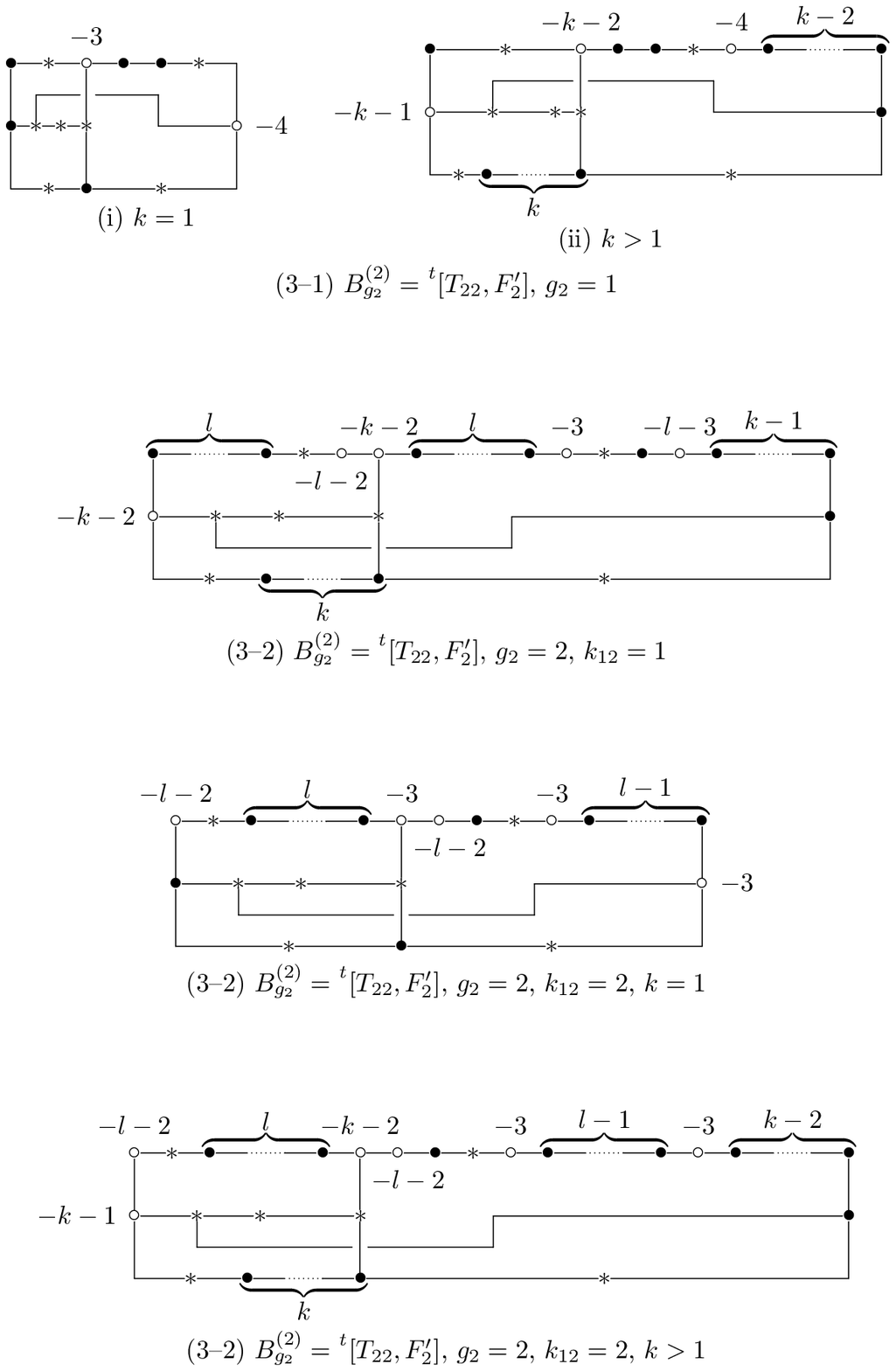}
\end{center}
\caption{The dual graph of $D+E_1+E_2$ --- continued}\label{figf2}
\end{figure}
\setcounter{figure}{3}
\begin{figure}
\begin{center}
\includegraphics{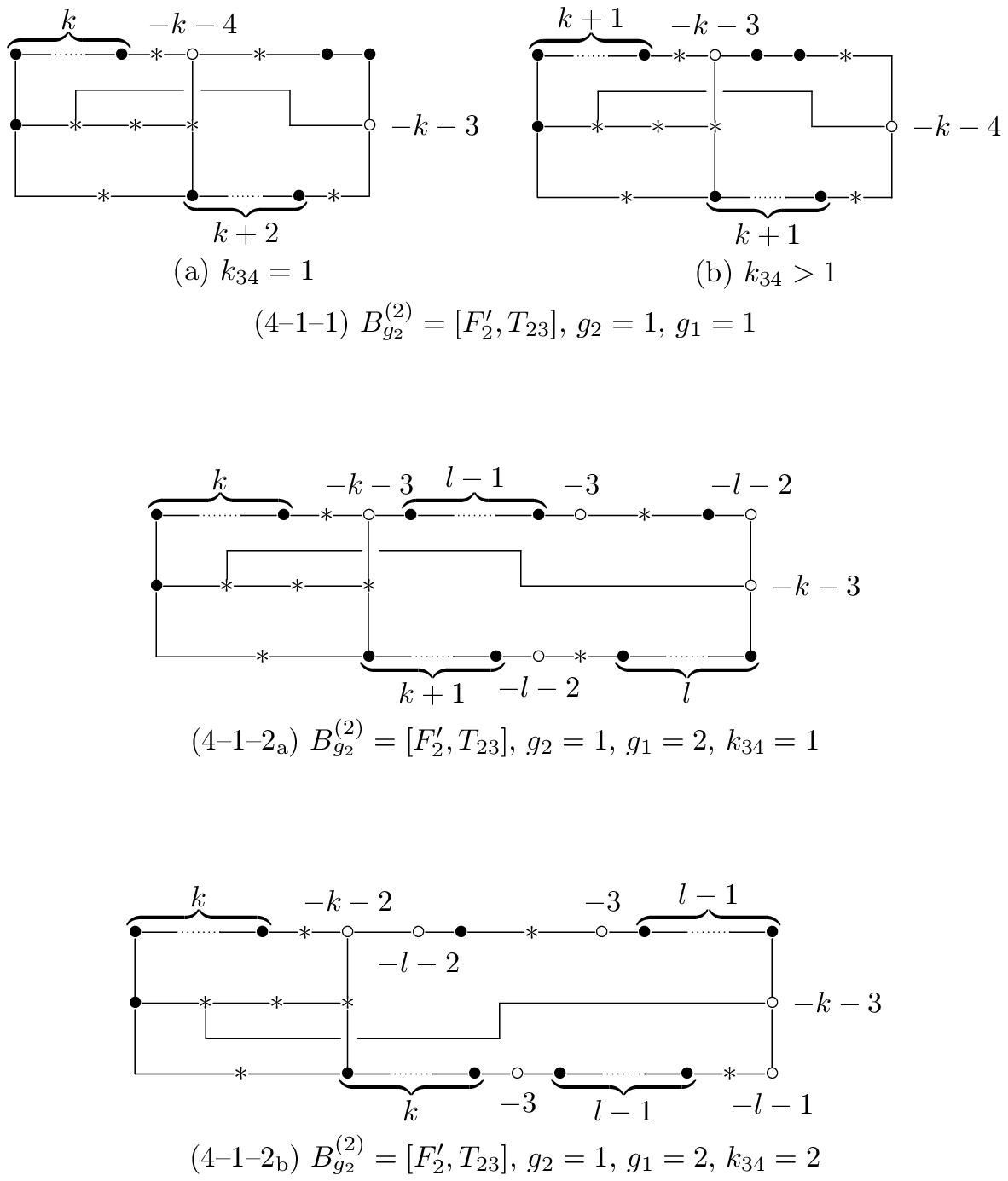}
\end{center}
\caption{The dual graph of $D+E_1+E_2$ --- continued}\label{figf3}
\end{figure}
\setcounter{figure}{3}
\begin{figure}
\begin{center}
\includegraphics{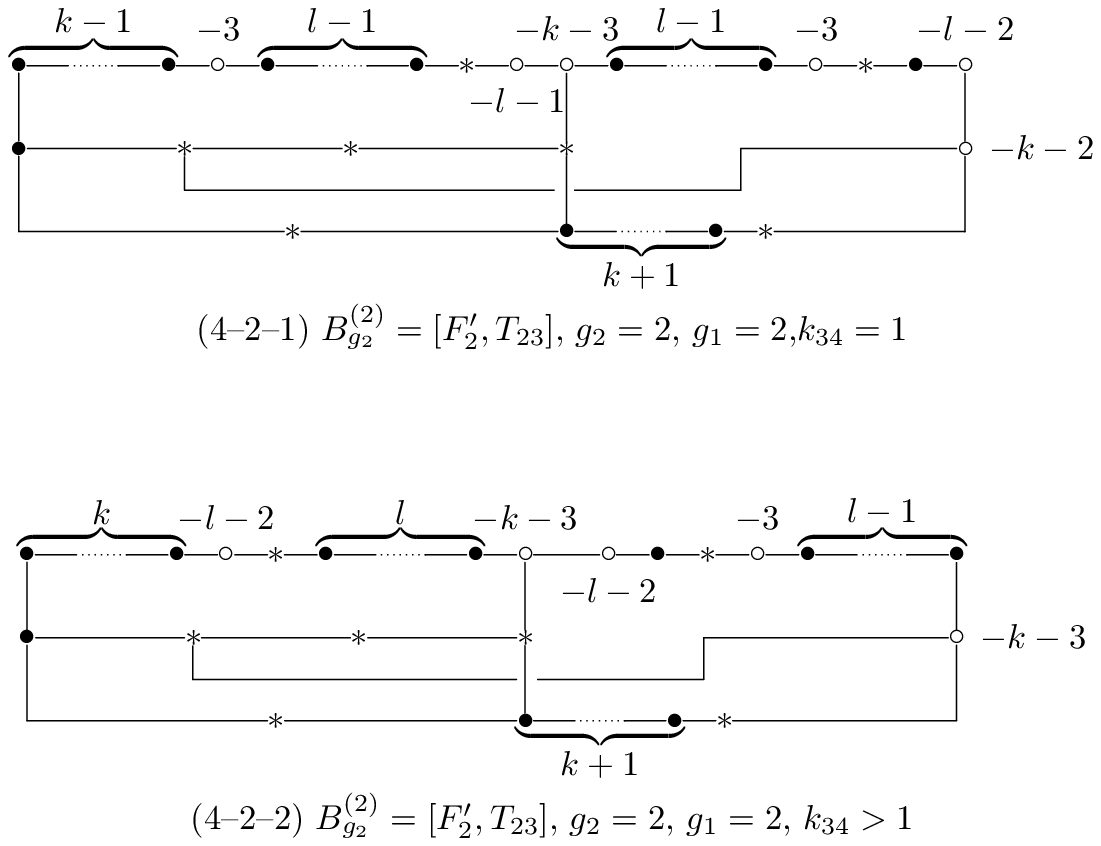}
\end{center}
\caption{The dual graph of $D+E_1+E_2$ --- continued}\label{figf4}
\end{figure}

We list
the dual graphs of $D+E_1+E_2$
in Figure~\ref{figf}.
We prove the converse assertion of Theorem~\ref{thm1}.
Let $\Gamma$ be one of the weighted dual graphs
in Figure~\ref{figf}.
It follows from \cite[Proposition 4.7]{fu} that
the sub-graphs $F_0$, $F_1$ and $F_2$ of $\Gamma$
can be contracted to three disjoint $0$-curves.
After the contraction,
$S_1$, $S_2$ and $S_3$ become disjoint $0$-curves
and meet with each curve $F_i$ transversally.
Thus
$\Gamma$ can be realized by blow-ups over
three sections and fibers of $\Sigma_0$.
By Lemma~\ref{lem:bu},
$\Gamma-E_1-E_2-C'$
can be contracted to two points of $\SP^2$.
Hence
all the numerical data in Theorem~\ref{thm1}
can be realized
as those of rational cuspidal plane curves.
\begin{acknowledgment}
The author would like to express his thanks to Professor Fumio Sakai
for his helpful advice.
The author was supported by the Fuujukai Foundation.
\end{acknowledgment}
%
%%%%%%%%%%%%%%%%%%%%%%%%%%%%%%%%%%%%%%%%%%%%%%%%%%%%%%%%%%%%%%%%%%%%

\ \\
\textsc{%
{\small
Department of Mathematics,
Graduate School of Science and Engineering,
Saitama University,\\
Saitama-City, Saitama 338--8570,
Japan.}}\\
{\small
\textit{E-mail address}:  \texttt{ktono@rimath.saitama-u.ac.jp}
}

\begin{thebibliography}{MaSa}
  \bibitem[BK]{bk}
    Brieskorn, E., Kn\"{o}rrer, H.:
    Plane algebraic curves.
    Basel, Boston, Stuttgart:
    Birkh\"{a}user 1986.
  \bibitem[Fe]{fe}
    Fenske, T.:
    Rational 1- and 2-cuspidal plane curves,
    Beitr\"{a}ge zur Algebra und Geometrie \textbf{40}, (1999), 309--329.
  \bibitem[Fu]{fu}
    Fujita, T.:
    On the topology of non-complete algebraic surfaces,
    J. Fac. Sci. Univ. Tokyo \textbf{29}, (1982), 503--566.
  \bibitem[FZ1]{fz:def}
    Flenner, H., Zaidenberg, M.:
    $\SQ$-acyclic surfaces and their deformations,
    Contemp. Math. \textbf{162}, (1994), 143--208.
  \bibitem[FZ2]{fz:dd2}
     Flenner, H., Zaidenberg, M.:
     On a class of rational cuspidal plane curves,
     Manuscripta Math. \textbf{89}, (1996), 439--459.
  \bibitem[Ka]{ka:cls}
    Kawamata, Y.:
    On the classification of non-complete algebraic surfaces,
    Proc. Copenhagen, Lecture Notes in Math. \textbf{732}, (1979), 215--232.
  \bibitem[Ki]{kiz}
    Kizuka, T.:
    Rational functions of $\SC^{\ast}$-type on the two-dimensional complex projective space,
    T\^{o}hoku Math. J. \textbf{38}, (1986), 123--178.
  \bibitem[Ko]{ko}
    Kojima, H.:
    Complements of plane curves with logarithmic Kodaira dimension zero,
    J. Math. Soc. Japan \textbf{52}, (2000), 793--806.
  \bibitem[MaSa]{masa}
    Matsuoka, T., Sakai, F.:
    The degree of rational cuspidal curves,
    Math. Ann. \textbf{285}, (1989), 233--247.
  \bibitem[MiSu]{misu}
    Miyanishi, M., Sugie, T.:
    $\SQ$-homology planes with $\SC^{\ast\ast}$-fibrations,
    Osaka J. Math. \textbf{28}, (1991), 1--26.
  \bibitem[MT1]{mits}
    Miyanishi, M., Tsunoda, S.:
    Non-complete algebraic surfaces with logarithmic Kodaira dimension $-\infty$ and with non-connected boundaries at infinity,
    Japan. J. Math. \textbf{10}, No.~2, (1984), 195--242.
  \bibitem[MT2]{mits2}
    Miyanishi, M., Tsunoda, S.:
    Absence of the affine lines on the homology planes of general type,
    J.~Math.~Kyoto Univ.~\textbf{32}, (1992), 443--450.
  \bibitem[O]{or}
    Orevkov, S.~Yu.:
    On rational cuspidal curves I. Sharp estimate for degree via multiplicities,
    Math.~Ann.~\textbf{324}, (2002), 657--673.
  \bibitem[To]{to:orev}
    Tono, K.:
    On Orevkov's rational cuspidal plane curves,
    J. Math. Soc. Japan \textbf{64}, (2012), 365--385.
  \bibitem[Ts]{Ts}
    Tsunoda, S.:
    The complements of projective plane curves,
    RIMS-K\^{o}ky\^{u}roku \textbf{446}, (1981), 48--56.
  \bibitem[W]{Wak}
    Wakabayashi, I.:
    On the logarithmic Kodaira dimension of the complement of a curve in $\SP^2$,
    Proc. Japan Acad. \textbf{54}, Ser. A, (1978), 157--162.
\end{thebibliography}
\end{document}